\documentclass[11pt]{amsart}
\usepackage{palatino}
\usepackage{amsfonts,amscd}
\usepackage{graphicx}
\usepackage{amsmath,latexsym,amssymb,amsthm}
\usepackage{hyperref}
\usepackage{cite}
\usepackage{lscape,fancyhdr}
\usepackage{a4}

\newcounter{cnt1}
\newcounter{cnt2}
\newcounter{cnt3}
\newcommand{\blr}{\begin{list}{$($\roman{cnt1}$)$} {\usecounter{cnt1}
        \setlength{\topsep}{0pt} \setlength{\itemsep}{0pt}}}
\newcommand{\bla}{\begin{list}{$($\alph{cnt2}$)$} {\usecounter{cnt2}
        \setlength{\topsep}{0pt} \setlength{\itemsep}{0pt}}}
\newcommand{\bln}{\begin{list}{$($\arabic{cnt3}$)$} {\usecounter{cnt3}
                \setlength{\topsep}{0pt} \setlength{\itemsep}{0pt}}}
\newcommand{\el}{\end{list}}

\newtheorem{Thm}{Theorem}[section]

\newtheorem{Prop}[Thm]{Proposition}
\newtheorem{Def}[Thm]{Definition}
\newtheorem{Exm}[Thm]{Example}
\newtheorem{Rem}[Thm]{Remark}

\begin{document}

\title{Hom-Lie-Rinehart Algebras }
\author{Ashis Mandal and Satyendra Kumar Mishra}
\footnote{The research of S.K. Mishra is supported by CSIR-SPM fellowship grant 2013 .}

\begin{abstract}
We introduce hom-Lie-Rinehart algebras as an algebraic analogue of hom-Lie algebroids, and systematically  describe a cohomology complex by considering coefficient modules. We define the notion of extensions for hom-Lie-Rinehart algebras. In the sequel, we deduce a characterisation of low dimensional cohomology spaces in terms of the  group of automorphisms of certain abelian extension and the equivalence classes of those abelian extensions in the category of hom-Lie-Rinehart algebras, respectively.  We also construct a canonical example of hom-Lie-Rinehart algebra associated to a given Poisson algebra and an automorphism.
\end{abstract}
\maketitle
\section{Introduction}
The notion of Lie-Rinehart algebra plays an important role in many branches of mathematics. The idea of this notion goes back to the work of N. Jacobson to study certain field extensions. It is also appeared in some different names in several areas which includes differential geometry and differential Galois theory. In \cite{Mackenzie}, K. Mackenzie provided a list of fourteen different terms mentioned for this notion. Here we follow the term Lie-Rinehart algebra, which is due to J. Huebschmann. He viewed Lie-Rinehart algebras as an  algebraic counterpart of Lie algebroids defined over smooth manifolds.  His work on several aspects of this algebra is developed systematically through a series of articles namely \cite{Hueb1, Hueb2, Hueb3, Hueb4}.

There is a growing interest in twisted algebraic structures or hom-algebraic structures defined for classical algebras and Lie algebroids as well. In the following years, by considering a vector space and an endomorphism of it, the corresponding hom- algebraic structures are introduced for various classical algebras. The first appearance of hom-algebra was the notion of  hom-Lie algebra, in the context of some particular deformation called $ q$-deformations of Witt and Virasoro algebra of vector fields. The work of J. Hartwig, D. Larsson and S. Silvestrov defined the notion of hom-Lie algebras to describe the $q$-deformations using  $\sigma$-derivations in place of usual derivation (\cite{HLIE01}). In the sequel, many concepts and properties have been extended to this framework of hom-structures. The study of hom-associative algebras, hom-Poisson algebras, Non-commutative hom-Poisson algebras, hom-Leibniz algebras and most of the results analogous to the classical algebras followed in the works of J. Hartwig, D. Larsson, A. Makhlouf, S. Silvestrov, D. Yau and other authors (\cite{HLIE01}, \cite{Sheng}, \cite{HALG}, \cite{HALG2}, \cite{DefHLIE}, \cite{NtHOM} ). Moreover, O. Elchinger introduced the quantization of Hom-Poisson structures in his thesis \cite{Oliver}.

C. Laurent-Gengoux and J. Teles introduced hom-Lie algebroid in \cite{hom-Lie}, where they also mentioned that the definition is not very straightforward to figure out from the classical notion of Lie algebroids so that the corresponding hom-version can be described in a systematic manner. First, they follow the classical case, one-to-one correspondence between  Lie algebroids structures on a vector bundle $A$ over a smooth manifold $M$ and Gerstenhaber algebras structures on the exterior algebra of multisections $\Gamma(\wedge^{*}A)$.  This makes natural the idea of defining hom-Lie algebroids through a formulation of hom-Gerstenhaber algebra.
 On the other hand, there are canonically defined adjoint functors between the category of Lie-Rinehart algebras and the category of Gerstenhaber algebras (see \cite{Gersh} for details).  
 
In this paper, we define hom-Lie-Rinehart algebras as an algebraic analogue of hom-Lie algebroids, and  prove that there are canonical adjoint functors between the category of hom-Gerstenhaber algebras and the category of hom-Lie-Rinehart algebras. We consider modules over a hom-Lie-Rinehart algebra and also describe a cohomology complex by considering coefficient modules. We define extensions of hom-Lie-Rinehart algebras and show the correspondence with lower dimensional cohomology spaces as one can expect (in an analogy with the classical algebras). This discussion not only provides a notion of cohomology and its interpretation for a more generalised notion of hom-Lie-Rinehart algebras but also it shed light on important special cases like hom-Lie algebras, Lie-Rinehart algebras, hom-Lie algebroids and hom-Gerstenhaber algebras. For instance, by considering hom-Lie algebra in the category of hom-Lie-Rinehart algebras one can view the deformation cohomology of hom-Lie algebras (\cite{NtHOM}), deduce the correspondence with low dimensional cohomology spaces and certain class of extensions. Also, the cohomology for Lie-Rinehart algebras and the characterisation of extensions can also be deduced to the one already present in the existing literature \cite{Hueb1}. Algebraic properties of the left modules over hom-Lie-Rinehart algebra motivates the definition of a representation of a hom-Lie algebroid, which is a generalization of the known representation for Lie algebroids. A canonical hom-Gerstenhaber algebra structure associated to a hom-Lie algebra is given in \cite{hom-Lie} (for details see Example \ref{Hom-Ger2} in the next section). If we use the boundary operator of a hom-Lie algebra complex with trivial coefficients defined in \cite{HALG2}, then we find that this operator generates the hom-Gerstenhaber bracket given in the Example \ref{Hom-Ger2}. Therefore it motivates the formulation of an exact hom-Gerstenhaber algebras or hom-Batalin-Vilkovisky algebras.

The paper is organized as follows:
In Section $2$, we recall some preliminaries on hom- algebras and fix  notations which we follow in the later part of the paper. We define homomorphisms of hom-Gerstenhaber algebras to form the category of hom-Gerstenhaber algebras. In Section $3$, we introduce the notion of  hom-Lie-Rinehart algebras and present various natural examples of this notion. This includes some of those examples are arising from the  hom-structures known in the literature. Next we consider homomorphisms in order to form the  category of hom-Lie-Rinehart algebras. In a sequel, we show that there are canonically defined adjoint functors from the category of hom-Lie-Rinehart algebras to the category of hom-Gerstenhaber algebras. The notion of a module for a hom-Lie-Rinehart algebra appeared in Section $4$, and subsequently we introduce a cochain complex and cohomology of a hom-Lie- Rinehart algebra with coefficients in a module. In Section $5$, we consider extensions of a hom-Lie- Rinehart algebras and characterise the first and second cohomology spaces in terms of the group of automorphisms of an $A$-split abelian extension and the equivalence classes of $A$-split abelian extensions in the category of hom-Lie Rinehart algebras, respectively. Furthermore, central extensions of a hom-Lie-Rinehart algebra are also defined.  In Section $6$, we describe Hom-Lie-Rinehart algebras canonically associated with a given Poisson algebra equipped with a Poisson algebra automorphism. In the last section we present a discussion of some special cases which also shows the wide interests and further application of hom-Lie-Rinehart algebras.

\section{Preliminaries on hom-algebras}
In this section, we shall recall basic definitions concerning hom-algebra structure from the literature (\cite{HLIE01, Sheng, HALG, HALG2, DefHLIE, NtHOM}) in order to  fix notation and terminology will be needed throughout the paper.

 Let $R$ denote a commutative ring with unity and $\mathbb{Z}_+$ be the set of all non-negative integers. We will consider all modules, algebras and their tensor products over such a ring $R$ and all linear maps to be $R$-linear unless otherwise stated. 
\begin{Def}
A hom-Lie algebra is a triplet $(\mathfrak{g},[-,-],\alpha)$ where $\mathfrak{g}$ is an $R$-module equipped with a skew-symmetric $R$-bilinear map $[-,-]:\mathfrak{g}\times \mathfrak{g}\rightarrow \mathfrak{g}$ and a linear map $\alpha:\mathfrak{g}\rightarrow \mathfrak{g}$, satisfying $\alpha[x,y]=[\alpha(x),\alpha(y)]$ such that the hom-Jacobi identity holds:
\begin{equation}
[\alpha(x),[y,z]]+[\alpha(y),[z,x]]+[\alpha(z),[x,y]]=0 ~~~~\mbox{for all}~~x,y,z\in \mathfrak{g}.
\end{equation}
Furthermore, if $\alpha$ is an automorphism of the $R$-module $\mathfrak{g}$, then the hom-Lie algebra $(\mathfrak{g},[-,-],\alpha)$ is called a regular hom-Lie algebra.
\end{Def}
\begin{Exm}
Given a Lie algebra $(\mathfrak{g},[-,-])$ with a Lie algebra homomorphism $\alpha:\mathfrak{g}\rightarrow \mathfrak{g}$, we can define a hom-Lie algebra as the triplet $(\mathfrak{g},\alpha\circ [-,-],\alpha )$. The hom-Jacobi identity for the new bracket $\alpha \circ [-,-]$ is equivalent to the Jacobi identity for Lie bracket $[-,-]$ ( restricted to  image of the map $\alpha \circ \alpha$). A hom-Lie  algebra of this kind is called a hom-Lie algebra obtained by composition.
\end{Exm}
\begin{Def}\label{Rep-hom-Lie}
A representation of a hom-Lie algebra $(\mathfrak{g},[-,-],\alpha)$ on a $R$-module $V$ is a pair $(\theta, \beta)$ of $R$-linear maps $\theta: \mathfrak{g}\rightarrow \mathfrak{g}l(V)$ and $\beta: V\rightarrow V$ such that 
\begin{equation}
\theta(\alpha(x))\circ \beta= \beta\circ \theta(x),
\end{equation}
\begin{equation}
\theta([x,y])\circ \beta= \theta(\alpha(x))\circ\theta(y)-\theta(\alpha(y))\circ\theta(x).
\end{equation}
for all $x,y \in \mathfrak{g}$.
\end{Def}
\begin{Exm}
 For any integer s, we can define the $\alpha^s$-adjoint representation of the regular hom-Lie algebra $(\mathfrak{g},[-,-],\alpha)$ on $\mathfrak{g}$ by $(ad_s,\alpha)$, where
$$ad_s(g)(h)=[\alpha^s(g),h]~~~\mbox{for all}~g,h\in \mathfrak{g}.$$
\end{Exm}
\begin{Def}
A graded hom-Lie algebra is a triplet $(\mathfrak{g},[-,-],\alpha)$ consisting of a graded module $\mathfrak{g}=\oplus_{i\in \mathbb{Z}_+}\mathfrak{g}_i$, a graded skew-symmetric bilinear map of degree $-1$ denoted by  $[-,-]:\mathfrak{g}\otimes\mathfrak{g}\rightarrow \mathfrak{g}$ and  $\alpha:\mathfrak{g}\rightarrow \mathfrak{g}$ is a homomorphism of $(\mathfrak{g},[-,-])$ of degree $0$ such that the following graded version of hom-Jacobi identity holds;
$$(-1)^{(i-1)(k-1)}[\alpha(x),[y,z]]+(-1)^{(j-1)(i-1)}[\alpha(y),[z,x]]+(-1)^{(k-1)(i-1)}[\alpha(z),[x,y]]=0,$$ 
 for all $x\in \mathfrak{g}_i,~y\in \mathfrak{g}_j,~k\in \mathfrak{g}_k$.
\end{Def}

\begin{Def}
A purely hom-Poisson algebra is a quadruple $(A,\mu,[-,-],\alpha)$ consisting of an $R$-module $A$, a bilinear map $\mu:A\otimes A\rightarrow A$ and a bilinear bracket $[-,-]:A\otimes A\rightarrow A$ such that following hold.
\begin{enumerate}
\item $(A,[-,-],\alpha)$ is a hom-Lie algebra;
\item $(A,\mu)$ is an associative commutative algebra;
\item $[x,\mu(y,z)]=\mu(\alpha(y),[x,z])+\mu(\alpha(z),[x,y]) ~~\mbox{for all}~ x,y,z\in A.$
\end{enumerate}
\end{Def}
\begin{Exm}
Given a Poisson algebra $(A,\mu,[-,-])$ and a Poisson algebra homomorphism $\alpha:A\rightarrow A$, the quadruple $(A,\mu,\alpha\circ[-,-],\alpha)$ is a purely hom-Poisson algebra. 
\end{Exm}
\begin{Def}
A Gerstenhaber algebra is a triplet  $(\mathcal{A}=\oplus_{i\in\mathcal{Z}_+}\mathcal{A}_i,\wedge,[-,-])$ where $\mathcal{A}$ is a graded commutative associative $R$-algebra, and $[-,-]:\mathcal{A}\otimes \mathcal{A}\rightarrow \mathcal{A}$ is a bilinear map of degree $-1$ such that:
\begin{enumerate}
\item $(\mathcal{A},[-,-])$ is a graded Lie algebra.
\item The following Leibniz rule holds:
      $$[X,Y\wedge Z]=[X,Y]\wedge Z+(-1)^{(i-1)j}Y\wedge [X,Z], $$
      for all $X\in \mathcal{A}_i,~Y\in \mathcal{A}_j,~Z\in \mathcal{A}_k$.
\end{enumerate}
\end{Def}
If $(\mathcal{A}=\bigoplus_{i\in \mathbb{Z}_+}\mathcal{A}_i,\bigwedge_1,[-,-]_{\mathcal{A}})$ and $(\mathcal{B}=\bigoplus_{i\in \mathbb{Z}_+}\mathcal{B}_i,\bigwedge_2,[-,-]_{\mathcal{B}})$ are Gerstenhaber algebras, then a $0$-degree $R$-linear map $\Theta: \mathcal{A}\rightarrow \mathcal{B}$ is a homomorphism of Gerstenhaber algebras if it satisfies:
\begin{enumerate}
\item $\Theta[X,Y]_{\mathcal{A}}=[\Theta(X),\Theta(Y)]_{\mathcal{B}} ~~\mbox{for all}~X\in \mathcal{A}_i, Y\in \mathcal{A}_j,$ and 
\item $\Theta(X\wedge_1 Y)=\Theta(X)\wedge_2 \Theta(Y)~~\mbox{for all}~X\in \mathcal{A}_i, Y\in \mathcal{A}_j.$
\end{enumerate}
We denote simply by $GR$, the category of Gerstenhaber algebras and the Gerstenhaber algebra homomorphisms. 

Recall that a Lie algebroid on a smooth manifold $M$ is a vector bundle $A$ over a smooth manifold $M$ with a vector bundle map $\rho: A\rightarrow TM$ called the anchor map, and a bilinear map denoted by the bracket $[-,-]: \Gamma A\otimes \Gamma A \rightarrow \Gamma A$ such that $(\Gamma A, [-,-])$ is a Lie-algebra and for all $X,Y\in \Gamma A$ and $f\in C^{\infty}(M)$ we have $[X,fY]=f [X,Y]+ \rho(X)(f) Y$.

Let $\Gamma (\wedge^*A)$ denote the algebra of multi-sections with standard wedge product. Then there is a one-to-one correspondence between Lie algebroid structures on $A$ and Gerstenhaber algebra structures on $\Gamma (\wedge^*A)$, where the graded associative commutative algebra structure on $\Gamma ( \wedge^*A)$ is given by the standard wedge product. 

Let $A$ be an associative commutative $R$-algebra and $Der(A)$ denote the space of $R$-derivations of the algebra $A$ . Then $Der(A)$ is simultaneously an $A$-module and a Lie algebra with the commutator bracket. 
\begin{Def}\label{Lie-Rin}
A Lie-Rinehart algebra $L$ over ( an associative commutative $R$-algebra ) $A$ is a Lie algebra  over $R$ with an $A$-module structure and a $R$-module homomorphism $\rho:L\rightarrow Der(A)$, such that 
$\rho$ is simultaneously an $A$-module homomorphism and a Lie $R$-algebra homomorphism and 
$$ [x, ay]=a[x,y]+\rho(x)(a)y~~\mbox{ for all}~ x,y\in L,~a\in A.$$
 \end{Def}
Let $A$ and $B$ be associative commutative $R$-algebras. Suppose $L$ and $L^{\prime}$ are Lie-Rinehart algebras over $A$ and $B$, respectively. Then a pair of maps $(g,f)$ is a homomorphism of the Lie-Rinehart algebras $L$ and $L^{\prime}$ if $g: A\rightarrow B$ is an algebra homomorphism and $f: L\rightarrow L^{\prime}$ is an $R$-module homomorphism such that following compatibility conditions hold.
\begin{enumerate}
\item $f(a.x)=g(a).f(x)~~\mbox{for all}~x\in L,~a\in A,$
\item $f[x,y]_L=[f(x),f(y)]_{L^{\prime}} ~~\mbox{for all}~x,y\in L,$
\item $g(\rho_L(x)(a))=\rho_{L^{\prime}}(f(x))(g(a))~~\mbox{for all}~x\in L,~ a\in A.$
\end{enumerate}  
We will denote by $LR$ the category of Lie-Rinehart algebras and the Lie-Rinehart algebra homomorphisms .
\begin{Thm}\label{EQ}
(Theorem 5, \cite{Gersh}) For a given Gerstenhaber algebra $(\mathcal{A}=\oplus_{i\in\mathcal{Z}_+}\mathcal{A}_i,\wedge,[-,-])$, we have a Lie-Rinehart algebra structure on $\mathcal{A}_1$ over the algebra $\mathcal{A}_0$. In fact, this assignment gives a functor $F:GR\rightarrow LR$. Similarly, we can define a functor $G:LR\rightarrow GR$ assigning a Lie- Rinehart algebra $\mathcal{A}_1$ over $\mathcal{A}_0$ to a Gerstenhaber algebra structure on $(\wedge^*_{\mathcal{A}_0}\mathcal{A}_1,\wedge,[-,-]_G)$, where $\wedge^*_{\mathcal{A}_0}\mathcal{A}_1$ is the exterior algebra of $\mathcal{A}_1$ over $\mathcal{A}_0$ and $[-,-]_G$ is the Sch$\ddot{o}$uten bracket:
$$[x_1\wedge..\wedge x_n,y_1\wedge..\wedge y_m]=(-1)^{(m-1)(n-1)} \sum(-1)^{i+j}[x_i,y_j]\wedge x_1\wedge..\wedge\hat{x_i}\wedge..\wedge x_n\wedge y_1\wedge..\wedge\hat{y_j}\wedge..\wedge y_m$$ 
where $x_i,y_j\in \mathcal{A}_1$ for all $1\leq i\leq n,~ 1\leq j\leq m$. The functor $G:LR\rightarrow GR$ is left adjoint to the functor $F:GR\rightarrow LR$. 
\end{Thm}

\begin{Def}\label{DerS}
(\cite{hom-Lie,HLIE01, hom-Lie1} ) Given an associative commutative algebra $A$, an $A$-module $M$ and an algebra endomorphism $\phi: A\rightarrow A$, we call an $R$-linear map $\delta: A\rightarrow M$ a $\phi$-derivation of $A$ into $M$ if it satisfies the required identity; 
$$\delta(a b)=\phi(a) \delta(b)+\phi(b) \delta(a)~~\mbox{for all}~~a,b \in A.$$
 Let us denote by $Der_{\phi}(A)$ the set of $\phi$-derivations.
\end{Def} 
\begin{Rem}
Suppose $M$ is a smooth manifold, $\psi: M\rightarrow M $ is a smooth map and induced map $\psi^*: C^{\infty}(M)\longrightarrow C^\infty(M)$ is defined by composition. Then  the space of $\psi^*$-derivations of the algebra of smooth functions $C^{\infty}(M)$ into itself can be identified with the space of sections of the pull- back bundle of the tangent bundle $TM$. This is usually denoted by $\Gamma(\psi^!TM).$ 
\end{Rem}
\begin{Def}\label{hom-Lie}
\cite{hom-Lie}
 A hom-Lie algebroid is a quintuple $(\xi,\phi,[-,-],\rho,\alpha)$, where $\xi: A \rightarrow M $ is a vector bundle over a smooth manifold $M$,
  $\phi: M \rightarrow M$ is a smooth map,  $[-,-]:\Gamma(A)\otimes \Gamma(A)\rightarrow \Gamma(A)$ is a bilinear map, the map $\rho: \phi^{!}A\rightarrow \phi^!TM$ is called the anchor and $\alpha: \Gamma (A)\rightarrow\Gamma(A)$ is a linear map such that following conditions are satisfied.
\begin{enumerate}
\item $\alpha(f X)=\phi^*(f) \alpha(X)$ for all $X\in \Gamma(A),~f \in C^{\infty}(M)$;
\item The triplet $(\Gamma(A),[-,-],\alpha)$ is a hom-Lie algebra;
\item The following hom-Leibniz identity holds:
$$[X,fY]=\phi^*(f)[X,Y]+\rho(X)[f]\alpha(Y);$$
for all $X,Y\in \Gamma(A),~f \in C^{\infty}(M)$
\item The pair $(\rho,\phi^*)$ is a representation of $(\Gamma(A),[-,-],\alpha)$ on $C^{\infty}(M).$
\end{enumerate}

A hom-Lie algebroid $(\xi,\phi,[-,-],\rho,\alpha)$ is called regular (or invertible) if the map $\alpha: \Gamma (A)\rightarrow\Gamma(A)$ is an invertible map and the smooth map $\phi: M \rightarrow M$ is a diffeomorphism.
\end{Def}
\begin{Rem}
 Note that $\rho(X)[f]$ denotes  a function on $M$ given by $$\rho(X)[f](m)=\big<d_{\phi(m)}f,\rho_m(X_{\phi(m)}) \big>$$for $m\in M$. Here the map $\rho_m:(\phi^!A)_m\cong A_{\phi(m)}\rightarrow (\phi^!TM)_m\cong T_{\phi(m)}M$ is the anchor map evaluated at $m\in M$ and $X_{\phi(m)}$ is  image of the section $X\in\Gamma (A)$ at $\phi(m)\in M$.
\end{Rem}
\begin{Def}
\cite{hom-Lie} 
A hom-Gerstenhaber algebra is a quadruple $(\mathcal{A}=\oplus_{i\in\mathcal{Z}_+}\mathcal{A}_i,\wedge,[-,-],\alpha)$ where $\mathcal{A}$ is a graded commutative associative $R$-algebra, $\alpha$ is an endomorphism of $(A,\wedge)$ of degree $0$ and $[-,-]:\mathcal{A}\otimes \mathcal{A}\rightarrow \mathcal{A}$ is a bilinear map of degree $-1$ such that:
\begin{enumerate}
\item $(\mathcal{A},[-,-],\alpha)$ is a graded hom-Lie algebra.
\item The hom-Leibniz rule holds:
      $$[X,Y\wedge Z]=[X,Y]\wedge \alpha(Z)+(-1)^{(i-1)j}\alpha(Y)\wedge [X,Z], $$
      for all $X\in \mathcal{A}_i,~Y\in \mathcal{A}_j,~Z\in \mathcal{A}_k$.
\end{enumerate}
\end{Def}
\begin{Exm}\label{Ger}
Given a Gerstenhaber algebra $(\mathcal{A},[-,-],\wedge)$ and an endomorphism $\alpha:(\mathcal{A},[-,-],\wedge)\rightarrow (\mathcal{A},[-,-],\wedge)$, the quadruple $(\mathcal{A},\wedge,\alpha\circ [-,-],\alpha)$ is a hom-Gerstenhaber algebra.  
\end{Exm}
\begin{Exm}\label{Hom-Ger2}
Suppose  $(\mathfrak{g},[-,-],\alpha)$ is a hom-Lie algebra. We can define a hom-Gerstenhaber algebra $(\mathfrak{G}=\wedge^*\mathfrak{g},\wedge,[-,-]_{\mathfrak{G}},\alpha_{\mathfrak{G}})$, where 
$$[x_1\wedge\cdots\wedge x_n,y_1\wedge \cdots\wedge y_m]_{\mathfrak{G}}=\sum_{i=1}^n\sum_{j=1}^m (-1)^{i+j}[x_i,y_j]\wedge \alpha_G(x_1\wedge\cdots \hat{x_i}\wedge\cdots\wedge x_n\wedge y_1\wedge \cdots \hat{y_j}\wedge\cdots\wedge y_m)$$
for all $x_1,\cdots,x_n,y_1,\cdots,y_m\in \mathfrak{g}$ and 
$$\alpha_{\mathfrak{G}}(x_1\wedge\cdots\wedge x_n)=\alpha(x_1)\wedge\cdots \wedge \alpha(x_n).$$
See \cite{hom-Lie} for further details.
\end{Exm}
If $(\mathcal{A}=\bigoplus_{i\in \mathbb{Z}_+}\mathcal{A}_i,\bigwedge_1,[-,-]_{\mathcal{A}},\alpha)$ and $(\mathcal{B}=\bigoplus_{i\in \mathbb{Z}_+}\mathcal{B}_i,\bigwedge_2,[-,-]_{\mathcal{B}},\beta)$ are hom-Gerstenhaber algebras, then a $0$-degree $R$-linear map $\Theta: \mathcal{A}\rightarrow \mathcal{B}$ is a homomorphism of hom-Gerstenhaber algebras if it satisfies following conditions.
\begin{enumerate}
\item $\Theta[X,Y]_{\mathcal{A}}=[\Theta(X),\Theta(Y)]_{\mathcal{B}} ~~\mbox{for all}~X\in \mathcal{A}_i, Y\in \mathcal{A}_j,$
\item $\Theta(X\wedge_1 Y)=\Theta(X)\wedge_2 \Theta(Y)~~\mbox{for all}~X\in \mathcal{A}_i, Y\in \mathcal{A}_j,$
\item For any $X\in \mathcal{A}_i,~~\Theta(\alpha(X))=\beta(\Theta(X)).$
\end{enumerate}
Denote the category of hom-Gerstenhaber algebras by $hGR$.

In Theorem 4.4, \cite{hom-Lie}, it is proved that there is a one-to-one correspondence between hom-Lie algebroid structures on a vector bundle $A$ over a smooth manifold  $M$ and hom-Gerstenhaber algebra structures on the graded vector space $\Gamma (\wedge^*A)$, where the graded associative commutative algebra structure on $\Gamma (\wedge^*A)$ is given by wedge product.

\section{Hom-Lie-Rinehart algebra}

In this section, we introduce the notion of hom-Lie-Rinehart algebras as an algebraic analogue of  hom-Lie algebroid defined. In a sequel, we define homomorphisms of these algebras to form a category, and we show that there are canonically defined adjoint functors between this category and  the category of hom-Gerstenhaber algebras.

Let $R$ be a commutative ring with unity, $A$  be an associative commutative $R$-algebra, and $\phi:A\rightarrow A$ be an algebra endomorphism. We will consider $A$-module action to be faithful. In the definition of a Lie-Rinehart algebra $L$ over $A$, we need to consider the Lie algebra $Der(A)$ of derivations on $A$. Here, we will consider the space of $\phi$-derivations (Definition \ref{DerS})
to define a hom-Lie-Rinehart algebra over the pair $(A, \phi)$.
\begin{Def}\label{hom-LR}
A hom-Lie Rinehart algebra over $(A,\phi)$  is a tuple $(A,L,[-,-],\phi,\alpha,\rho)$ ,
 where $A$ is an associative commutative algebra, $L$ is an $A$-module, $[-,-]:L\times L\rightarrow L$ is a skew symmetric bilinear map, the map $\phi:A\rightarrow A$ is an algebra homomorphism, $\alpha:L\rightarrow L $ is a linear map satisfying $\alpha([x,y])=[\alpha(x),\alpha(y)]$, and the $R$-linear map $\rho: L\rightarrow Der_{\phi}A$  are such that following conditions hold.
\begin{enumerate}
\item The triplet $(L,[-,-],\alpha)$ is a hom-Lie algebra.
\item $\alpha(a.x)=\phi(a).\alpha(x)$ for all $a\in A,~x\in L $.
\item $(\rho, \phi)$ is a representation of $(L,[~,~],\alpha)$ on $A$.
\item $\rho(a.x)=\phi(a).\rho(x)$ for all $a\in A,~x\in L $.
\item $[x,a.y]=\phi(a)[x,y]+ \rho(x)(a)\alpha(y)$ for all $a\in A,~x,y\in L $.
\end{enumerate}
\end{Def}
A \textbf{hom-Lie-Rinehart algebra} $(A,L,[-,-],\phi,\alpha,\rho)$ is said to be regular if the map $\phi:A\rightarrow A$ is an algebra automorphism and $\alpha:L\rightarrow L $ is a bijective map.
\begin{Exm}
A \textbf{Lie-Rinehart algebra} $L$  over $A$ with the Lie bracket $[-,-]:L\otimes L\rightarrow L$ and the anchor map $\rho:L\rightarrow Der(A)$ is a hom-Lie-Rinehart algebra $(A,L,[-,-],\phi,\alpha,\rho)$ where $\alpha=Id_L,~\phi=Id_A$ and $\rho:L\rightarrow Der_{\phi}A=Der(A)$.
\end{Exm}
If we consider the map $\alpha=Id_L$ in the above definition, then a hom-Lie-Rinehart algebra is a Lie-Rinehart algebra. In fact, $\alpha=Id_L$ forces $\phi=Id_A$, because of the identity: $\alpha(a.x)=\phi(a).(\alpha(x))$ for all $a\in A,~ x\in L$ and the fact that the action of $A$ on $L$ is faithful. 
\begin{Exm}\label{h-LA}
A \textbf{hom-Lie algebra} $(L,[-,-],\alpha)$ structure over an $R$-module $L$ gives the hom-Lie Rinehart algebra $(A,L,[-,-],\phi,\alpha,\rho)$ with $A=R$, the algebra morphism $\phi=Id_{\mathbb{K}}$ and the trivial action of $L$ on $R$. 
\end{Exm}
\begin{Exm}\label{h-Lie-algebroid}
Let us recall the definition of \textbf{hom-Lie algebroid $(A\rightarrow M,[-,-],\rho,\alpha,\phi)$} from Definition \ref{hom-Lie}. It provides a hom-Lie-Rinehart algebra $(C^{\infty}(M),\Gamma A,[-,-],\phi^*,\alpha,\rho)$ where $\Gamma A$ is the space a sections of the underline vector bundle $A\rightarrow M$ and $\phi^*:C^{\infty}(M)\rightarrow C^{\infty}(M)$ is canonically defined by the smooth map $\phi: M\rightarrow M$.
\end{Exm}
\begin{Exm}\label{pderhl}
Assume that $\phi: A\rightarrow A$ is an automorphism. Then $(Der_{\phi}A,[-,-]_{\phi},\alpha_{\phi})$ is a hom-Lie algebra. 
Here, the bilinear map
 $[-,-]_{\phi}:Der_{\phi}A\times Der_{\phi}A\rightarrow Der_{\phi}A$ is given by
\begin{equation}\label{rel1}
[D_1,D_2]_{\phi}=\phi\circ D_1\circ \phi^{-1}\circ D_2 \circ \phi^{-1}-\phi\circ D_2\circ \phi^{-1}\circ D_1 \circ \phi^{-1}.
\end{equation}
Also, the linear map $\alpha_{\phi}:Der_{\phi}A\rightarrow Der_{\phi}A$ is defined as
\begin{equation}\label{rel2}
\alpha_{\phi}(D)=\phi\circ D\circ \phi^{-1}.
\end{equation}

Furthermore, the space of $\phi$-derivations $Der_{\phi}A$ is an $A$-module (action defined using algebra multiplication in $A$) satisfying following identities:
\begin{enumerate}
\item $\alpha(a.D)=\phi(a).\alpha(D)$ for $a\in A,~\mbox{and} ~D\in Der_{\phi}A$;
\item $[D_1,a.D_2]_{\phi}=\phi(a).[D_1,D_2]_{\phi}+\alpha_{\phi}(D_1)(a).\alpha_{\phi}(D_2)$ for $a\in A,~\mbox{and} ~D_1,D_2\in Der_{\phi}A$.
\end{enumerate}

In turn it follows that the tuple $(A,Der_{\phi}A,[-,-]_{\phi},\alpha_{\phi},\alpha_{\phi})$ is a hom-Lie-Rinehart algebra over $(A,\phi)$ where the anchor map $\rho=\alpha_{\phi}$.

\end{Exm}
\begin{Exm}\label{hom-LR byc}
If we consider a Lie-Rinehart algebra $L$ over $A$ along with an endomorphism $$(\phi,\alpha):(A,L)\rightarrow (A,L)$$ in the category of Lie-Rinehart algebras then we get a hom-Lie-Rinehart algebra $(A,L,[-,-]_{\alpha}, \phi,\alpha,\rho_{\phi})$, called \textbf{``obtained by composition"}. Here
\begin{enumerate}
\item $[x,y]_{\alpha}=\alpha[x,y]$ for $x,y\in L$; 
\item $\rho_{\phi}(x)(a)=\phi(\rho(x)(a))$ for $x\in L, ~a\in A$.
\end{enumerate} 
Note that the following conditions are satisfied.
\begin{enumerate}
\item $(L,[-,-]_{\alpha},\alpha)$ is a hom-Lie algebra obtained by composition;
\item $\alpha(a.x)=\phi(a).\alpha(x)$;
\item $\rho_{\phi}(a.x)=\phi(a).\rho_{\phi}(x)$;
\item $[x,a.y]_{\alpha}=\phi(a).[x,y]_{\alpha}+\rho_{\phi}(x)(a).\alpha(y)$;
\item $(\rho_{\phi},\phi)$ is a representation of $(L,[~,~]_{\alpha},\alpha)$ on $A$;
\end{enumerate}
for all $x,y\in L,~a\in A$.
\end{Exm}
\begin{Exm}\label{T Hom-LR algebra}
Let $(L,[-,-],\alpha)$ be a hom-Lie algebra over $R$ and $A$ be an associative commutative $R$-algebra with a homomorphism $\phi: A\rightarrow A$, and $(\rho,\phi)$ be a representation of $(L,[-,-],\alpha)$ on $A$. Furthermore, the map $\rho$ to be a $R$-linear map from $L\rightarrow Der_{\phi}(A)$ gives the action of $L$ on $A$ via $\phi$-derivations. Then the \textbf{transformation hom-Lie-Rinehart algebra} structure on $\mathfrak{g}=A\otimes L$ is given by the tuple $(A,\mathfrak{g},[-,-],\phi,\tilde{\alpha},\tilde{\rho})$. More precisely we have the following:
\begin{enumerate}
\item The $R$-bilinear bracket $[-,-]$ on $\mathfrak{g}$ is given by
$$[a\otimes x, b\otimes y]_{\mathfrak{g}}= \phi(ab)\otimes [x,y]+ \phi(a)\rho(x)(b)\alpha(y)- \phi(b)\rho(y)(a)\alpha(x),$$
for all $x,y\in L$ and $a,b \in A$.
\item The $R$-linear map $\tilde{\alpha}:\mathfrak{g}\rightarrow \mathfrak{g}$ is given by $$\tilde{\alpha}(a\otimes x)= \phi(a)\otimes \alpha (x),$$
for all $a\in A,~x\in L$. Here, $(\mathfrak{g},[~,~]_{\mathfrak{g}},\tilde{\alpha})$ is a hom-Lie algebra and $\tilde{\alpha}[g,h]=[\tilde{\alpha}(g),\tilde{\alpha}(h)]$ for all $g,h\in \mathfrak{g}$.
\item The action of $\mathfrak{g}$ on $A$ via $\phi$-derivations is given by $$\tilde{\rho}(a\otimes x)(b)= \phi(a)\rho(x)(b)$$ for all $x\in L,~a,b\in A$.
\end{enumerate}
Note that in general, transformation hom-Lie-Rinehart algebras are not obtained by composition.
\end{Exm}
\begin{Exm}\label{Product}
Let $(A,L,[-,-]_L,\phi,\alpha_L,\rho_L)$ and $(A,M,[-,-]_M,\phi,\alpha_M,\rho_M)$ be hom-Lie-Rinehart algebras over $(A,\phi)$. We consider $$L\times_{Der_{\phi}A}M=\{(l,m)\in L\times M : \rho_l(l)=\rho_M(m)\},$$ where $L\times M$ denotes the Cartesian product. Then $(A,L\times_{Der_{\phi}A}M,[-,-],\phi,\alpha,\rho)$ is a hom-Lie-Rinehart algebra, where  
\begin{enumerate}
\item the bracket is given by 
$$[(l_1,m_1),(l_2,m_2)]=([l_1,l_2],[m_1,m_2]);$$
\item the endomorphism $\alpha:L\times_{Der_{\phi}A}M\rightarrow L\times_{Der_{\phi}A}M$ is given by
$$\alpha(l,m)=(\alpha_L(l),\alpha_M(m));$$
\item and the anchor map $\rho: L\times_{Der_{\phi}A}M\rightarrow Der_{\phi}A$ is given by
$$\rho(l,m)(a)=\rho_L(a)=\rho_M(a);$$
\end{enumerate}   
for all $l,l_1,l_2\in L$, $m,m_1,m_2\in M$, and $a\in A$. The above  structure gives the categorical product in the category $hLR_A^{\phi}$. Note that cartesian product is not the product in this category as expected from the case of Lie-Rinehart algebras \cite{CentExt}. 
\end{Exm}
\begin{Rem}
Suppose $\mathcal{A}=\bigoplus_{i\in\mathbb{Z}} \mathcal{A}_i$ is  a graded $R$-module. Let us recall from  Example \ref{Ger} that given a Gerstenhaber algebra structure $(\mathcal{A},\wedge,[-,-])$ and an endomorphism $\alpha:(\mathcal{A},\wedge,[-,-])\rightarrow (\mathcal{A},\wedge,[-,-])$,  we have a hom-Gerstenhaber algebra $(\mathcal{A},\wedge,[-,-]_{\alpha},\alpha)$. Later we will see that for a given hom-Gerstenhaber algebra $(\mathcal{A}=\bigoplus_{i\in\mathbb{Z}} \mathcal{A}_i,[-,-],\wedge,\alpha)$,  there is a  hom-Lie-Rinehart algebra structure  $(\mathcal{A}_0,\mathcal{A}_1,[-,-]_1,\phi,\alpha_{\mathcal{A}_1},\rho)$. In other words, given any Gerstenhaber algebra and an endomorphism of the Gerstenhaber algebra, we have a canonical hom-Lie-Rinehart algebra.
\end{Rem}
Next we define homomorphisms of hom-Lie-Rinehart algebras.
\begin{Def}
Let $(A,L,[-,-]_{L},\phi,\alpha_L,\rho_L)$ and $(B,L^{\prime},[-,-]_{L^{\prime}},\psi,\alpha_{L^{\prime}},\rho_{L^{\prime}})$ be hom-Lie-Rinehart algebras, then a hom-Lie-Rinehart algebra homomorphism is defined as a pair of maps $(g,f)$, where the map $g:A\rightarrow B$ is a $R$-algebra homomorphism and $f: L_1\rightarrow L_2$ is a $R$-linear map such that following identities hold:
\begin{enumerate}
\item $f(a.x)=g(a).f(x)~~\mbox{for all}~x\in L_1,~a\in A,$
\item $f[x,y]_L=[f(x),f(y)]_{L^{\prime}} ~~\mbox{for all}~x,y\in L,$
\item $f(\alpha_L(x))=\alpha_{L^{\prime}}(f(x)) ~~\mbox{for all}~x\in L,$
\item $g(\phi(a))=\psi(g(a))~~\mbox{for all}~a\in A$
\item $g(\rho_L(x)(a))=\rho_{L^{\prime}}(f(x))(g(a))~~\mbox{for all}~x\in L,~ a\in A.$
\end{enumerate}  
\end{Def} 
Hom-Lie-Rinehart algebras with homomorphisms form a category  of hom-Lie-Rinehart algebras, which we by $hLR$. If $\alpha_L=Id_L$ and $\alpha_{L^{\prime}}=Id_{L^{\prime}}$ (which in turn give $\phi=Id_A$ and $\psi=Id_B$.), then the above hom-Lie Rinehart algebras are the usual Lie-Rinehart algebras and the homomorphism of the hom-Lie Rinehart algebras is the usual homomorphism in category of Lie-Rinehart algebras. In fact, the category of Lie-Rinehart algebras is a full subcategory of the category of hom-Lie-Rinehart algebras. 
\begin{Rem}
If $A=B$ and $\phi=\psi$, then by taking $g=Id_A: A\rightarrow A$ we get homomorphism of hom-Lie- Rinehart algebras over $(A,\phi)$. Let us denote by $hLR^{\phi}_A$ the category of hom-Lie Rinehart algebras over $(A,\phi)$. To simplify the notations we denote a hom-Lie Rinehart algebra $(A,L,[-,-],\phi,\alpha,\rho)$ over $(A,\phi)$ simply by $(\mathcal{L},\alpha)$ and similarly for any other hom-Lie Rinehart algebra over $(A,\phi)$, say $(A,L^{\prime},[-,-]^{\prime},\phi,\alpha^{\prime},\rho^{\prime})$ simply by the notation $(\mathcal{L}^{\prime},\alpha^{\prime})$.
\end{Rem}
\subsection{Hom-Gerstenhaber algebras and hom-Lie -Rinehart algebras}
Given a hom-Gerstenhaber algebra $(\mathcal{A}=\bigoplus_{i\in\mathbb{Z}} \mathcal{A}_i,[-,-],\wedge,\alpha)$, we have the following identities for the pair $(\mathcal{A}_0,\mathcal{A}_1)$ of $R$-modules:
\begin{enumerate}
\item $\mathcal{A}_0$ is a commutative $R$-algebra and there is an $\mathcal{A}_0$-module action on $\mathcal{A}_1$, where commutative product is given by $\wedge:\mathcal{A}_0\otimes \mathcal{A}_0\rightarrow \mathcal{A}_0$ and action is given by $\wedge:\mathcal{A}_0\otimes \mathcal{A}_1\rightarrow \mathcal{A}_1.$(using the fact that $\wedge$ is a graded commutative product on $\mathcal{A}$.)

\item $(\mathcal{A}_1,[-,-]_1,\alpha_1)$ is a hom-Lie algebra, where $[-,-]_1:\mathcal{A}_1\otimes \mathcal{A}_1\rightarrow \mathcal{A}_1$ is the restriction of $[-,-]:\mathcal{A}\otimes \mathcal{A}\rightarrow \mathcal{A}$ on $\mathcal{A}_1\otimes \mathcal{A}_1$ and $\alpha_1:\mathcal{A}_1\rightarrow \mathcal{A}_1$ is restriction of $\alpha:\mathcal{A}\rightarrow\mathcal{A}$ on $\mathcal{A}_1$. Here, hom-Jacobi identity follows from graded hom-Jacobi identity for the graded bracket $[-,-].$

\item By graded hom-Jacobi identity, $[x,a\wedge b]=[x,a]\wedge \alpha_0(b)+ \alpha_0(a)\wedge [x,b]$ for all $x\in \mathcal{A}_1,~ a,b \in \mathcal{A}_0.$

\item $[a\wedge x, b]=\phi(a)\wedge [x,b]$ for all $x\in \mathcal{A}_1,~ a,b \in \mathcal{A}_0.$ Here we are using the graded hom-Jacobi identity and the fact that $[a,b]=0$.

\item $[x, a\wedge y]=[x,a]\wedge \alpha_1 (y)+ \alpha_0 (a) \wedge [x,y]$ for all $x,y\in \mathcal{A}_1,~ a \in \mathcal{A}_0.$
\end{enumerate}
Thus we have a hom-Lie-Rinehart algebra $(\mathcal{A}_0,\mathcal{A}_1,[-,-]_1,\phi,\alpha_{\mathcal{A}_1},\rho)$, where $\phi:=\alpha_0,~\alpha_{\mathcal{A}_1}:=\alpha_1,$ and $\rho:L\rightarrow Der_{\phi}(\mathcal{A}_0)$ is given by $\rho(x)(a)=[x,a]$ for all $x\in \mathcal{A}_1,~a\in \mathcal{A}_0$. (The definition of $\rho$ makes sense because of the identity $(3)$. Also by identity $(4)$, we have $\rho(ax)(b)=\phi(a)\rho(x)(b)$ for all $a,b\in \mathcal{A}_0,~ x\in \mathcal{A}_1$.) 

Furthermore, it follows
 that there is a functor
 $$\mathcal{F}:hGR\rightarrow hLR$$
 which assigns  the hom-Lie- Rinehart algebra $(\mathcal{A}_0,\mathcal{A}_1,[-,-]_1,\phi,\alpha_{\mathcal{A}_1},\rho)$ to a hom-Gerstenhaber algebra $(\mathcal{A}=\bigoplus_{i\in\mathbb{Z}} \mathcal{A}_i,[-,-],\wedge,\alpha)$.

The following theorem gives a one-to-one correspondence between  hom-Gerstenhaber algebra structures on $(\wedge_A^*L,[-,-]_G,\wedge,\alpha)$ and  hom-Lie Rinehart algebra structures on $(A,L,[-,-],\phi,\alpha,\rho)$. This result generalises the well known one-to-one correspondence between Lie-Rinehart algebras and Gerstenhaber algebras.
\begin{Thm}\label{EQ1}
Suppose, $hLR$ and $hGR$ denote the categories of hom-Lie-Rinehart algebras and hom-Gerstenhaber algebras, respectively. Let $\mathcal{G}: hLR\rightarrow hGR$ assigns $$(A,L,[-,-],\phi,\alpha,\rho) \mapsto (\wedge_A^* L,\wedge,[-,-]_G,\alpha_G),$$ where $\alpha_G: \wedge_A^*L\rightarrow \wedge_A^*L$ given by: 
\begin{equation}
\alpha_G(aX_1\wedge\cdots\wedge X_n)=\phi(a)\alpha(X_1)\wedge\cdots\wedge \alpha(X_n),
\end{equation}
and $[~,~]_G:\wedge_A^* L\otimes \wedge_A^* L\rightarrow \wedge_A^* L$ is given by:
$$[X_1\wedge\cdots\wedge X_n,a]_G=\sum_{i=1}^n \rho(X_i)(a)\wedge\alpha_G(X_1\wedge\cdots\wedge\hat{X_i}\wedge\cdots\wedge X_n)$$
and by 
$$[X_1\wedge\cdots\wedge X_n,Y_1\wedge\cdots\wedge Y_m]_G=\sum_{i=1}^n \sum_{j=1}^m[X_i,Y_j]\wedge\alpha_G(X_1\wedge\cdots\wedge\hat{X_i}\wedge\cdots\wedge X_n\wedge Y_1\wedge\cdots\wedge\hat{Y_j}\wedge\cdots\wedge Y_m) $$
where $X_i,Y_j\in L,~a\in A;~~1\leq i\leq n, ~1\leq j\leq m.$
Then $\mathcal{G}$ is a functor from the category $hLR$ to the category $hGR$. Moreover the functor $\mathcal{G}$ is left adjoint to the functor $\mathcal{F}: hGR\rightarrow hLR$. 
\end{Thm}
\begin{proof}
We denote the Gerstenhaber algebra $(\mathcal{A},\wedge_{\mathcal{A}},[-,-]_{\mathcal{A}},\alpha_{\mathcal{A}})$ simply by $\mathcal{A}$.
Given a homomorphism $\Theta:\mathcal{A}\rightarrow\mathcal{B}$ in the category $hGR$, we have a homomorphism of hom-Lie Rinehart algebras $\mathcal{F}(\Theta):\mathcal{F}(\mathcal{A})\rightarrow \mathcal{F}(\mathcal{B})$, where $\mathcal{F}(\Theta)$ is obtained by restricting the map $\Theta $ on $(\mathcal{A}_0,\mathcal{A}_1)$. Similarly, given a homomorphism $\Phi: (\mathcal{L},\alpha)\rightarrow (\mathcal{M},\beta)$ in the category $hLR$, we get the morphism $\mathcal{G}(\Phi):\wedge_A^*L \rightarrow \wedge^*_A\mathcal{M}$ in the category $hGR$, which is defined by extending the map $\Phi$ from $(A,L)$ to $\wedge_A^*L$(extension of the map is similar to the extension of $\alpha_G$ from $\alpha$). Now, from the definition of homomorphisms in the categories $hGR$ and $hLR$, for every pair of $(\mathcal{L},\alpha)\in hLR$ and $\mathcal{A}\in hGR$ it follows that 
$$Hom_{hGR}(\mathcal{G}(\mathcal{L},\alpha), \mathcal{A})\cong Hom_{hLR}((\mathcal{L},\alpha), \mathcal{F}(\mathcal{A})).$$
Moreover the bijection is functorial in $(\mathcal{L},\alpha)\in hLR$ and $\mathcal{A}\in hGR$. In fact, for every morphism $f:(\mathcal{L},\alpha)\rightarrow(\mathcal{L}^{\prime},\alpha^{\prime})$ and $g:\mathcal{A}\rightarrow \mathcal{A}^{\prime}$, by considering $g_*$ and $f^*$ as the respective induced maps, we have the following commutative diagram:
$$\begin{CD}
Hom_{hGR}(\mathcal{G}(\mathcal{L}^{\prime},\alpha^{\prime}), \mathcal{A}) @> \mathcal{G}(f)^* >>
Hom_{hGR}(\mathcal{G}(\mathcal{L},\alpha), \mathcal{A}) @> g_{*} >> Hom_{hGR}(\mathcal{G}(\mathcal{L},\alpha), \mathcal{A}^{\prime})\\
@V\cong VV @V\cong VV @V\cong VV\\
Hom_{hLR}((\mathcal{L}^{\prime},\alpha^{\prime}),\mathcal{F} (\mathcal{A})) @> f^* >>
Hom_{hLR}((\mathcal{L},\alpha), \mathcal{F}(\mathcal{A})) @> \mathcal{F}(g)_{*} >> Hom_{hLR}((\mathcal{L},\alpha), \mathcal{F}(\mathcal{A}^{\prime}))
\end{CD}$$
 Consequently, $\mathcal{G}: hLR\rightarrow hGR$ is left adjoint to the functor $\mathcal{F}:hGR\rightarrow hLR$. 
\end{proof}
\begin{Rem}\label{CompS}
Here, we have considered the hom-Lie algebroids defined in \cite{hom-Lie} to define the algebraic counterpart as hom-Lie-Rinehart algebras. There is a modified definition of hom-Lie algebroids appeared in  \cite{hom-Lie1}. But, if we follow this modified definition then the one-to-one correspondence in  Theorem \ref{EQ1} will not be available any more. 
\end{Rem}

\section{Cohomology of  hom-Lie-Rinehart algebra}
We now define cohomology of a hom-Lie-Rinehart algebra. First we define the notion of (left-) module over a hom-Lie-Rinehart algebra.
\subsection{Modules over hom-Lie-Rinehart algebras}
Let $A$ be an associative and commutative $R$-algebra and $\phi$ be an algebra automorphism of $A$ and $(\mathcal{L},\alpha)$ be a hom-Lie-Rinehart algebra over $(A, \phi)$.
\begin{Def}\label{Mod}
Let $M$ be an $A$-module, and $\beta\in End_{R}(M)$. Then the pair $(M,\beta)$ is a left module over a hom-Lie Rinehart algebra $(\mathcal{L},\alpha)$ if the following holds.
\begin{enumerate}
\item There is a map $\theta:L\otimes M\rightarrow M$, such that the pair  $(\theta,\beta)$ is a representation of the hom-Lie algebra $(L,[-,-],\alpha)$ on $M$. Let us denote $\theta(x,m)$ by $\{x,m\}$ for $x\in L,~m\in M$.
\item $\beta(a.m)=\phi(a).\beta(m)$ for all $a\in A~\mbox{and}~m\in M$.
\item $\{a.X,m\}=\phi(a)\{X,m\}$ for all $a\in A,~X\in L,~m\in M$.
\item $\{X,a.m\}=\phi(a)\{X,m\}+\rho(X)(a).\beta(m)$ for all $X\in L,~a\in A,~m\in M$.
\end{enumerate} 
\end{Def}
If we consider $\alpha=Id_L$ and $\beta=Id_M$, then $(\mathcal{L},\alpha)$ is a Lie-Rinehart algebra and $M$ is a left Lie-Rinehart algebra module over the Lie-Rinehart algebra $L$.
\begin{Exm}
For $\alpha=Id_L$ any left Lie-Rinehart algebra module $M$ over $(A, L,[-,-], \rho)$ gives a left hom-Lie-Rinehart algebra module $(M,Id_M)$ over $(\mathcal{L},\alpha)$.  
\end{Exm}
\begin{Exm}
The pair $(A,\phi)$ is a left module over $(\mathcal{L},\alpha)$. As $(\rho,\phi)$ is a representation of $(L,[-,-],\alpha)$ over $A$. Further the conditions $(3)$ and $(4)$ are satisfied by definition of the map  $\rho$.
\end{Exm}

\subsection{Cochain complex of a Hom-Lie Rinehart algebra}
Let $(\mathcal{L},\alpha)$ be a hom-Lie Rinehart algebra over $(A,\phi)$ and $(M,\beta)$ be a module over $(\mathcal{L},\alpha)$. We consider the $\mathbb{Z}_+$-graded space of $R$-modules 
$$C^*(L;M):=\oplus_{n\geq 1}C^n(L;M)$$
 for hom-Lie-Rinehart algebra $(\mathcal{L},\alpha)$ with coefficients in $(M,\beta)$, where $C^n(L;M)\subseteq Hom_R(\wedge_R^n L,M)$ consisting of elements $f\in Hom_R(\wedge_R^n L,M)$ satisfying conditions below.
\begin{enumerate}
\item $f(\alpha(x_1),\cdots,\alpha(x_n))=\beta(f(x_1,x_2,\cdots,x_n))$ for all $x_i\in L,~1\leq i\leq n$
\item $f(x_1,\cdots,a.x_i,\cdots,x_n)=\phi^{n-1}(a)f(x_1,\cdots,x_i,\cdots,x_n)$ for all $x_i\in L,~1\leq i\leq n,~\mbox{and}~ a\in A$.
\end{enumerate}
Define the $R$-linear maps $\delta:C^n(L;M)\rightarrow C^{n+1}(L;M) $ given by 
\begin{equation}
\begin{split}
\delta f(x_1,\cdots,x_{n+1}):= &\sum_{i=1}^{n+1}(-1)^{i+1}\{\alpha^{n-1}(x_i),f(x_1,\cdots,\hat{x_i},\cdots,x_{n+1})\}\\&+\sum_{1\leq i<j\leq n+1}f([x_i,x_j],\alpha(x_1),\cdots,\hat{\alpha(x_i)},\cdots,\hat{\alpha(x_j)},\cdots,\alpha(x_{n+1}))
\end{split}
\end{equation}
for all $f\in C^n(L;M),~ x_i\in L$, where $1\leq i\leq n+1$. Here we follow these notations and deduce next that the map $\delta$  gives rise to a coboundary map.
\begin{Prop}
If $f\in C^n(L;M)$, then $\delta f \in C^{n+1}(L;M)$ and $\delta^2=0$. 
\end{Prop}
\begin{proof}
First, we need to check that $\delta f(\alpha(x_1),\alpha(x_2)\cdots,\alpha(x_{n+1}))=\beta(\delta f(x_1,x_2,\cdots,x_{n+1}))$ for all $x_i\in L,~1\leq i\leq n+1$. We will use the fact that $f\circ \alpha= \beta \circ f$ and $\{\alpha(x),\beta(m)\}=\beta\{x,m\}$ (as $f\in C^n(L;M)$ and $(\theta,\beta)$ is a representation of $(L,[-,-],\alpha)$ on $M$). Now
\begin{align*}
\delta f(\alpha(x_1),\cdots,\alpha(x_{n+1}))&= \sum_{i=1}^{n+1}(-1)^{i+1}\{\alpha^{n}(x_i),f(\alpha(x_1),\cdots,\hat{\alpha(x_i)},\cdots,\alpha(x_{n+1}))\}\\
&+\sum_{1\leq i<j\leq n+1}f(\alpha([x_i,x_j]),\alpha^2(x_1),\cdots,\hat{\alpha^2(x_i)},\cdots,\hat{\alpha^2(x_j)},\cdots,\alpha^2(x_{n+1}))\\
&=\sum_{i=1}^{n+1}(-1)^{i+1}\{\alpha(\alpha^{n-1}(x_i)),\beta \big(f(x_1,\cdots,\hat{x_i},\cdots,x_{n+1})\big)\}\\
&+\sum_{1\leq i<j\leq n+1}\beta\big(f([x_i,x_j],\alpha(x_1),\cdots,\hat{\alpha(x_i)},\cdots,\hat{\alpha(x_j)},\cdots,\alpha(x_{n+1}))\big)\\
&=\sum_{i=1}^{n+1}(-1)^{i+1}\beta\Big(\{\alpha^{n-1}(x_i),f(x_1,\cdots,\hat{x_i},\cdots,x_{n+1})\}\Big)\\
&+\sum_{1\leq i<j\leq n+1}\beta\big(f([x_i,x_j],\alpha(x_1),\cdots,\hat{\alpha(x_i)},\cdots,\hat{\alpha(x_j)},\cdots,\alpha(x_{n+1}))\big)\\
&=\beta(\delta f(x_1,x_2,\cdots,x_{n+1})).\\
\end{align*}
Also, we need to check the expression
$\delta f(x_1,\cdots,a.x_i,\cdots,x_{n+1})=\phi^n(a)\delta f(x_1,\cdots,x_i,\cdots,x_{n+1})$ for all $x_i\in L,~1\leq i\leq n+1~\mbox{and}~a\in A$. But it follows from the simple calculation and using the fact that $(\rho,\phi)$ is a representation of $(L,[-,-],\alpha)$ on $A$. (i.e. $\rho(\alpha(x))(\phi(a))=\phi(\rho(x)(a))$ for all $x\in L, ~a\in A$.) 

Further, $\delta^2=0$ follows from the direct but a long calculation.
\end{proof}
By the above proposition, $(C^*(L,M),\delta)$ is a cochain complex. The resulting cohomology of the cochain complex we define to be the cohomology space of hom-Lie-Rinehart algebra $(\mathcal{L},\alpha)$ with coefficients in $(M,\beta)$, and we denote this cohomology as $H^*_{hLR}(L,M)$.

 We will use this cohomology in the next section when we consider extensions of a hom-Lie-Rinehart algebras. 
\begin{Rem}
Let $M$ be a smooth manifold. Let $A$ denote the space of smooth functions $C^{\infty}(M)$ and $L=\chi(M)$, the space of smooth vector fields on $M$, with $\alpha=Id_L$. Then the above cochain complex $(C^*(L,A),\delta)$ with coefficients in $(A,Id_A)$ is the de-Rham complex of $M$ (except $0$-cochains) and for $n\geq 2$, the $n^{th}$-cohomology group $H^n_{hLR}(L,A)$ is same as the $n^{th}$ de-Rham cohomology group of $M$. 
\end{Rem}

\section{Extensions of hom-Lie-Rinehart algebras}
 In this section, we introduce extensions of a hom-Lie-Rinehart algebra
  and we follow the same notations as in previous sections. First note that the category $hLR_A^{\phi}$ does not have zero object. Thus, by a short exact sequence written as
$$\begin{CD}
(\mathcal{L}^{\prime\prime},\alpha^{\prime\prime}) @>i>> (\mathcal{L}^{\prime},\alpha^{\prime}) @>\sigma>> (\mathcal{L},\alpha)@.
\end{CD}$$ in the category $hLR_A^{\phi}$ what we mean is that the homomorphism $i:(\mathcal{L}^{\prime\prime},\alpha^{\prime\prime})\rightarrow(\mathcal{L}^{\prime},\alpha^{\prime})$ is injective, the homomorphism $\sigma:(\mathcal{L}^{\prime},\alpha^{\prime})\rightarrow (\mathcal{L},\alpha)$ is surjective and $\sigma\circ i=0$. 

\begin{Def} 
A short exact sequence in the category $hLR^{\phi}_A$
$$\begin{CD}
(\mathcal{L}^{\prime\prime},\alpha^{\prime\prime}) @>i>> (\mathcal{L}^{\prime},\alpha^{\prime}) @>\sigma>> (\mathcal{L},\alpha)@.
\end{CD}$$
is called an extension of the hom-Lie-Rinehart algebra $(\mathcal{L},\alpha)$ by the hom-Lie-Rinehart algebra $(\mathcal{L}^{\prime\prime},\alpha^{\prime\prime})$. Here, anchor map of the hom-Lie-Rinehart algebra $(\mathcal{L}^{\prime\prime},\alpha^{\prime\prime})$ is trivial, i.e. $\rho^{\prime\prime}=0$, since $\sigma\circ i=0$.
\end{Def}

An extension of hom-Lie-Rinehart algebra $(\mathcal{L},\alpha)$ is said to be $A$-split if we have an $A$-module map $\tau: (\mathcal{L},\alpha)\rightarrow (\mathcal{L}^{\prime},\alpha^{\prime})$ such that
\begin{enumerate}
\item $\sigma \circ \tau=Id_{(\mathcal{L},\alpha)}$;
\item $\tau(a.x)=a.\tau(x)$ and
\item $\tau\circ\alpha = \alpha^{\prime}\circ \tau$
\end{enumerate}
for each $a\in A,~x\in L$. Furthermore, if the section $\tau$ for the map $\sigma$ is a homomorphism of hom-Lie-Rinehart algebras, then this extension is said to be split in the category of hom-Lie-Rinehart algebras.
\begin{Def}\label{Remark}
Let $(\mathcal{L},\alpha)=(A,L,[-,-]_1,\phi,\alpha,\rho_1)$ be a hom-Lie-Rinehart algebra over $(A,\phi)$ and 
$(\mathcal{M},\beta)=(A,M,[-,-]_2,\phi,\beta,\rho_2)$ be another hom-Lie-Rinehart algebra  over $(A,\phi)$ with anchor map $\rho_2=0$, then we say that $(\mathcal{L},\alpha)$ acts on $(\mathcal{M},\beta)$ if the following conditions hold.
\begin{enumerate}
\item There is a map $\theta:L\otimes M\rightarrow M$, such that the pair  $(\theta,\beta)$ is a representation of the hom-Lie algebra $(L,[-,-]_1,\alpha)$ on $M$. Let us denote $\theta(x,m)$ by $\{x,m\}$ for $x\in L,~m\in M$.
\item For $x\in L,~m,n\in M$, we have $\{\alpha(x),[m,n]_2\}=[\{x,m\},\beta(n)]_2+[\beta(m),\{x,n\}]_2$, where $\{-,-\}:L\otimes M\rightarrow M$ is the action of $L$ on $M$.
\item $\{a.X,m\}=\phi(a)\{X,m\}$ for all $a\in A,~X\in L,~m\in M$.
\item $\{X,a.m\}=\phi(a)\{X,m\}+\rho(X)(a).\beta(m)$ for all $X\in L,~a\in A,~m\in M$.
\end{enumerate}
\end{Def}
In the above definition, condition $(1)$ and $(2)$ imply that $(L,[-,-]_1,\alpha)$ acts on $(M,[-,-]_2,\beta)$ in the category of hom-Lie $R$-algebras.

Now, let us consider $(\mathcal{L}^{\prime},\alpha^{\prime}):=(A,L^{\prime},[-,-]^{\prime},\phi,\alpha^{\prime},\rho^{\prime})$, where 
\begin{itemize}
\item $L^{\prime}=L\oplus M$, direct sum of $A$-modules;
\item $[(x,m),(y,n)]^{\prime}=([x,y]_1,[m,n]_2+\{x,n\}-\{y,m\}); $ 
\item $\alpha^{\prime}((x,m))=(\alpha(x),\beta(m));$
\item $\rho^{\prime}(x,m)=\rho(x),$
\end{itemize}
for all $x,y\in L,~m,n\in M $. Then $(L^{\prime},[-,-]^{\prime}, \alpha^{\prime})$ is a hom-Lie algebra and it is the semi-direct product of $(L,[-,-]_1,\alpha)$ and $(M,[-,-]_2,\beta)$ in the category of hom-Lie $R$-algebras. Also there is an $A$-module structure on $L^{\prime}$. This yields that $(\mathcal{L}^{\prime},\alpha^{\prime})=(A,L^{\prime},[-,-]^{\prime},\phi,\alpha^{\prime},\rho^{\prime})$ is a hom-Lie-Rinehart algebra, which we denote by $(\mathcal{L},\alpha)\rtimes (\mathcal{M},\beta)$. 

In particular, if $\phi=id_A,~\beta=id_M$, and $\alpha=id_L$, then we have
\begin{enumerate}
\item[(i)]
the action defined above is an action of a Lie-Rinehart algebra $L$ (over $A$) on a Lie $A$-algebra $M$ and
\item[(ii)]the semi direct product of $L$ and $M$ in category of Lie $R$-algebras gives a Lie-Rinehart algebra over $A$ ( for details see Section 2, \cite{CentExt}).
\end{enumerate}
Note that if $[-,-]_M=0$, then Definition \ref{Mod} becomes a particular case of the above definition since any hom-Lie-Rinehart algebra module $(M,\beta)$ is a hom-Lie-Rinehart algebra over $(A,\phi)$ (with a trivial bracket and a trivial action on $A$).

\begin{Rem}
Not every extension of a hom-Lie-Rinehart algebra is $A$-split. Assume that $L$ is a projective $A$-module, then we have an $A$-linear map $\tau:L\rightarrow L^{\prime}$ such that $\sigma\circ\tau=Id_L$, but it may not satisfy the identity: $\tau\circ\alpha = \alpha^{\prime}\circ \tau$. However, if we take a hom-Lie-Rinehart algebra $(\mathcal{L},\alpha)$ over $(A,\phi)$ acting on a hom-Lie-Rinehart algebra  $(\mathcal{M},\beta)$ over $(A,\phi)$ with trivial anchor map, then for $(\mathcal{L}^{\prime},\alpha^{\prime})=(\mathcal{L},\alpha)\rtimes (\mathcal{M},\beta)$ the short exact sequence
$$\begin{CD}
(\mathcal{M},\beta)@>i>> (\mathcal{L}^{\prime},\alpha^{\prime}) @> \sigma>>(\mathcal{L},\alpha)
\end{CD}$$
is split in the category $hLR^{\phi}_A$. Thus, it is also an $A$-split extension of $(\mathcal{L},\alpha)$. 
\end{Rem}

\subsection{Abelian Extensions:}
In this subsection, we will define an abelian extension of a hom-Lie-Rinehart algebra by a module in the category $hLR^{\phi}_A$. Note that any hom-Lie-Rinehart algebra module $(M,\beta)$ gives a hom-Lie-Rinehart algebra $(A,M,[-,-]_M,\phi,\beta,\rho_M) \in hLR^{\phi}_A$, with a trivial bracket and a trivial anchor map. Let us denote this object in $hLR^{\phi}_A$ by $(\mathcal{M},\beta)$.
\begin{Def}
Let $(\mathcal{L},\alpha)$ be a hom-Lie-Rinehart algebra over $(A,\phi)$ and $(M,\beta)$ be a module over $(\mathcal{L},\alpha)$. A short exact sequence 
$$\begin{CD}
(\mathcal{M},\beta)@>i>> (\mathcal{L}^{\prime},\alpha^{\prime}) @> \epsilon>>(\mathcal{L},\alpha)@. 
\end{CD}$$
in the category $hLR^{\phi}_A$, is called an abelian extension of $(\mathcal{L},\alpha)$ by $(M,\beta)$ if 
$$[i(m),x]=i((\epsilon(x)).m)~~\mbox{for all }m\in M, ~x\in L^{\prime}.$$
\end{Def}
Next, we will show that the second cohomology space $H^2_{hLR}(L,M)$ of a hom-Lie-Rinehart algebra $(\mathcal{L},\alpha)$ with coefficients in $(M,\beta)$ classifies $A$-split abelian extensions of $(\mathcal{L},\alpha)$ by $(M,\beta)$.

This result generalises the well-known classification theorems for the classical cases of a Lie algebras \cite{HomAlg} and  Lie-Rinehart algebras \cite{Hueb1}.

\begin{Thm}\label{Char1}
There is a one-to-one correspondence between the equivalence classes of $A$-split abelian extensions of a hom-Lie-Rinehart algebra $(\mathcal{L},\alpha)$ by $(M,\beta)$ and the cohomology classes in $H^2_{hLR}(L,M)$.
\end{Thm}

\begin{proof}
Let $f$ be a representative of the cohomology class $[f]\in H^2_{hLR}(L,M)$. Consider a hom-Lie-Rinehart algebra $(\mathcal{L}^{\prime},\alpha^{\prime}):=(A,L^{\prime},[-,-]^{\prime},\phi,\alpha^{\prime},\rho^{\prime})$, where the structure constraints are given as follows: 
\begin{enumerate}
\item $L^{\prime}=L\oplus M$, a direct sum of $A$-modules;
\item $[(x,m),(y,n)]^{\prime}=([x,y],[x,n]-[y,m]+f(x,y)); $ 
\item $\alpha^{\prime}((x,m))=(\alpha(x),\beta(m));$
\item $\rho^{\prime}(x,m)=\rho(x)=\rho(\pi(x,m)),$
\end{enumerate}
for all $x,y\in L,~m,n\in M $ and $\pi:L^{\prime}\rightarrow L$ defined as $\pi(x,m)=x$. Then 
$$\begin{CD}
(\mathcal{M},\beta)@>i>> (\mathcal{L}^{\prime},\alpha^{\prime}) @> \pi>>(\mathcal{L},\alpha),
\end{CD}$$
is an A-split abelian extension of $(\mathcal{L},\alpha)$ by $(M,\beta)$, where $i:M\rightarrow L^{\prime}$ is defined by $i(m)=(0,m)$. 

Suppose we take an another representative $f^{\prime}$ of the cohomology class $[f]\in H^2_{hLR}(L,M)$ and get an extension $(\mathcal{L}^{\prime\prime},\alpha^{\prime\prime})$ as above. Since $f$ and $f^{\prime}$ represent the same cohomology class $[f]$, we have $f-f^{\prime}=\delta g $ for some $g\in C^1(L,M)$. Then the map $F: (\mathcal{L}^{\prime},\alpha^{\prime})\rightarrow (\mathcal{L}^{\prime\prime},\alpha^{\prime\prime}) $ defined by $F(x,m)=(x,m+g(x))$ gives an isomorphism of the above extensions obtained by using $f$ and $f^{\prime}$ respectively. Thus for a cohomology class in $H^2_{hLR}(L,M)$ there is a unique equivalence class of $A$-split abelian extensions of $(\mathcal{L},\alpha)$ by $(M,\beta)$.

Conversely, let  
$$\begin{CD}
(\mathcal{M},\beta)@>i>> (\mathcal{L}^{\prime},\alpha^{\prime}) @> \sigma>>(\mathcal{L},\alpha)
\end{CD}$$
be an A-split abelian extension of the hom-Lie- Rinehart algebra $(\mathcal{L},\alpha)$ by $(M,\beta)$. We will first show that we can define a $2$-cocycle in $C^2(L,M)$ which is independent of a section for the map $\sigma$.

First, we fix a section $\tau: L\rightarrow L^{\prime}$ for the map $\sigma$. Now consider the map $G:L\oplus M\rightarrow L^{\prime}$ given by 
$$G(x,m)=\tau(x)+i(m).$$  
Then it follows that $G$ is an injective $A$-module homomorphism. In fact, $G$ is an isomorphism of $A$-modules.

Define a $2$-cochain $ \Omega_{\tau}\in C^2(L,M)$ by the following expression;
\begin{equation*}
\Omega_{\tau}(x,y)=i^{-1}\big([\tau(x),\tau(y)]-\tau([x,y])\big),
\end{equation*}
for all $x,y\in L$. Here we have
\begin{enumerate}
\item $\Omega_{\tau}$ is a skew-symmetric $R$-bilinear map and it satisfies $\Omega_{\tau}(a.x,y)=\phi(a)\Omega_{\tau}(x,y)$  for all $x,y\in L,~a\in A$;
\item $\delta(\Omega_{\tau})=0$, which follows using hom-Jacobi identity for $(L^{\prime},[-,-]^{\prime},\alpha^{\prime})$;
\item $\Omega_{\tau}\circ\alpha=\beta\circ\Omega_{\tau}$, which follows by the relations $\tau\circ\alpha=\alpha^{\prime}\circ\tau,~$ and $\alpha^{\prime}\circ i=i\circ \beta$.
\end{enumerate}
Consequently,  we get that $\Omega_{\tau}$ is a $2$-cocycle in $C^2(L,M)$.  

Note that if we take another section $\tau^{\prime}:L\rightarrow L^{\prime}$ of $\sigma$. Then the resulting 2-cocycle $\Omega_{\tau^{\prime}}$ is cohomologous to $\Omega_{\tau}$. This follows from the fact that $\Omega_{\tau^{\prime}}-\Omega_{\tau}=\delta(i^{-1}\circ(\tau^{\prime}-\tau))$ for $(i^{-1}\circ(\tau^{\prime}-\tau))\in C^1(L,M)$. Thus for a given $A$-split abelian extension of $(\mathcal{L},\alpha)$ by $(M,\beta)$, there exists a unique cohomology class $[\Omega_{\tau}]\in H^2(L,M).$

In order to complete the proof, we need to show that for two equivalent $A$-split abelian extensions, the associated $2$-cocycles are cohomologous.

Let 
$$\begin{CD}
(\mathcal{M},\beta)@>i^{\prime}>> (\mathcal{L}^{\prime\prime},\alpha^{\prime\prime}) @> \sigma^{\prime}>>(\mathcal{L},\alpha)
\end{CD}$$
be another $A$-split abelian extension of $(\mathcal{L},\alpha)$ by $(M,\beta)$ , and it is isomorphic to the extension:
$$\begin{CD}
(\mathcal{M},\beta)@>i>> (\mathcal{L}^{\prime},\alpha^{\prime}) @> \sigma>>(\mathcal{L},\alpha). 
\end{CD}$$
Suppose the map $\Phi:(\mathcal{L}^{\prime},\alpha^{\prime})\rightarrow(\mathcal{L}^{\prime\prime},\alpha^{\prime\prime}) $ is an isomorphism of these extensions, that is the following diagram commutes:
$$\begin{CD}
(\mathcal{M},\beta)@>i>> (\mathcal{L}^{\prime},\alpha^{\prime}) @> \sigma>>(\mathcal{L},\alpha)\\
  @|  @V \Phi VV @|\\
(\mathcal{M},\beta)@>i^{\prime}>> (\mathcal{L}^{\prime\prime},\alpha^{\prime\prime}) @> \sigma^{\prime}>>(\mathcal{L},\alpha)
\end{CD}$$ 
 Now we will show that for a section $\tau:(\mathcal{L},\alpha)\rightarrow (\mathcal{L}^{\prime},\alpha^{\prime})$ of $\sigma$ and $\tau^{\prime}:(\mathcal{L},\alpha)\rightarrow (\mathcal{L}^{\prime\prime},\alpha^{\prime\prime})$ of $\sigma^{\prime}$, the respective associated cocycles $\Omega_{\tau}$ and $\Omega_{\tau^{\prime}}$ are cohomologous. Consider, $\tau^{\prime\prime}=\Phi\circ \tau :(\mathcal{L},\alpha)\rightarrow (\mathcal{L}^{\prime\prime},\alpha^{\prime\prime})$ a section for $\sigma^{\prime}.$ Then we have $\Omega_{\tau^{\prime\prime}}=\Omega_{\tau}.$  Here, $\Omega_{\tau^{\prime\prime}}$ and $\Omega_{\tau^{\prime}}$ are cohomologous in $H^2_{hLR}(L,M)$ because of the fact that $\tau^{\prime}$ and $\tau^{\prime\prime}$, both are sections of $\sigma^{\prime}$. Therefore, $\Omega_{\tau}$ and $\Omega_{\tau^{\prime}}$ are cohomologous in $H^2_{hLR}(L,M)$.
\end{proof}
\begin{Rem}\label{Altbr}
Consider an $A$-split extension of a hom-Lie-Rinehart algebra $(\mathcal{L},\alpha)$ by a module $(M,\beta)$:
$$\begin{CD}
(\mathcal{M},\beta)@>i>> (\mathcal{L}^{\prime},\alpha^{\prime}) @> \sigma>>(\mathcal{L},\alpha).
\end{CD}$$ 
If we fix a section $\tau$ for the map $\sigma$ then we have an isomorphism of the underlying $A$-modules given by  $G: L^{\prime}\rightarrow L\oplus M$ where $G(X)=(\sigma(X),i^{-1}(X-\tau\circ\sigma(X))).$ Let us denote by notation $(x,m)_{\tau}$
the inverse image of $(x,m)$ under the isomorphism $G$, which is $\tau(x)+i(m)$. Then any $X,Y\in L^{\prime}$ can be written as $X=(x,m)_{\tau}$ and $Y=(y,n)_{\tau}$ for some $x,y\in L$ and $m,n\in M$. Moreover, the Lie bracket on $L^{\prime}$ can  be expressed as 
$$[X,Y]^{\prime}=([x,y],[x,n]-[y,m]+\Omega_{\tau}(x,y))_{\tau}.$$
\end{Rem}
In the next result we will present a characterisation of the first cohomology space $H^1_{hLR}(L,M)$ in terms of group of automorphisms of an $A$-split abelian extension.

\begin{Thm}\label{Cor2}
There is a one-to-one correspondence between the group of automorphisms of a given $A$-split abelian extension,
$$\begin{CD}
(\mathcal{M},\beta)@>i>> (\mathcal{L}^{\prime},\alpha^{\prime}) @> \sigma>>(\mathcal{L},\alpha)
\end{CD}$$
of a hom-Lie-Rinehart algebra $(\mathcal{L},\alpha)$ by $(M,\beta)$ ($(\mathcal{M},\beta)$ is corresponding object in $hLR_A^{\phi}$) and cohomology space $H^1_{hLR}(L,M)$.
\end{Thm}

\begin{proof}
Let $F: (\mathcal{L}^{\prime},\alpha^{\prime}) \rightarrow (\mathcal{L}^{\prime},\alpha^{\prime}) $ be a hom-Lie-Rinehart algebra isomorphism which gives an automorphism of the extension
$$\begin{CD}
(\mathcal{M},\beta)@>i>> (\mathcal{L}^{\prime},\alpha^{\prime}) @> \sigma>>(\mathcal{L},\alpha)~~~.
\end{CD}$$
So we have the following commutative diagram:
$$\begin{CD}
(\mathcal{M},\beta)@>i>> (\mathcal{L}^{\prime},\alpha^{\prime}) @> \sigma>>(\mathcal{L},\alpha)\\
  @|  @V F VV @|\\
(\mathcal{M},\beta)@>i>> (\mathcal{L}^{\prime},\alpha^{\prime}) @> \sigma>>(\mathcal{L},\alpha).
\end{CD}$$ 
Since the given short exact sequence is $A$-split,  we have a section $\tau:(L,\alpha)\rightarrow (L^{\prime},\alpha^{\prime})$ (see Remark 2.12.) For the section $\tau$, we get an isomorphism of underlying $A$-modules, say $G: L^{\prime}\rightarrow L\oplus M$. Let us denote the inverse image $\tau(x)+i(m)$ of $(x,m)$ under the isomorphism, by $(x,m)_{\tau}$ for all $x\in L,~m\in M$.

Assume $F_1,~F_2$ are maps obtained by taking projections of the map $G\circ F$ onto first and second components. Then $F(x,m)_{\tau}=(F_1(x,m)_{\tau},F_2(x,m)_{\tau})_{\tau}$ for $(x,m)_{\tau}\in L^{\prime}$. By the commutativity of the above diagram $F_1((x,m)_{\tau})=x$ and $F_2((0,m)_{\tau})=m$. Therefore, for $(x,m)_{\tau}\in L^{\prime}$,
$$(x,F_2((x,0)_{\tau})+m)_{\tau}=(x,F_2((x,m)_{\tau}))_{\tau}.$$

Define $\psi:L\rightarrow M$ by $\psi(x)=F_2((x,0)_{\tau})$. So, we can  write $F_2((x,m)_{\tau})=\psi(x)+m$. Note that $\psi(a.x)=a.\psi(x)$ for all $x\in L,~a\in A$ and $\psi\circ\alpha=\beta\circ\psi$, i.e. $\psi$ is a 1-cochain in $C^1(L,M)$. Now, $F: (\mathcal{L}^{\prime},\alpha^{\prime}) \rightarrow (\mathcal{L}^{\prime},\alpha^{\prime}) $ is an isomorphism of hom-Lie Rinehart algebras, i.e. we have 
$$F[(x,0)_{\tau},(y,0)_{\tau}]^{\prime}=[F(x,0)_{\tau},F(y,0)_{\tau}]^{\prime}.$$

Using the Remark \ref{Altbr}, we get $\delta\psi=0.$
Therefore $\psi$ is a 1-cocycle representing a cohomology class in $H^1_{hLR}(L,M)$. 

Conversely, assume that $\psi$ represents a cohomology class in $H^1_{hLR}(L,M)$.

Define $F:(\mathcal{L}^{\prime},\alpha^{\prime}) \rightarrow (\mathcal{L}^{\prime},\alpha^{\prime})$ as 
$F((x,m)_{\tau})=(x,m+\psi(x))_{\tau}.$

Note that the following identities are satisfied:
\begin{enumerate}
\item $F:L^{\prime}\rightarrow L^{\prime}$ is an $A$-module homomorphism. 

\item $\alpha^{\prime}\circ F=F\circ\alpha^{\prime}$, which follows by using the equations $\psi\circ\alpha=\beta\circ\psi$ and $\alpha^{\prime}\circ i = i\circ\beta.$ 

\item $F[(x,m)_{\tau},(y,n)_{\tau}]^{\prime}=[F(x,m)_{\tau},F(y,n)_{\tau}]$, which follows from the condition $\delta \psi(x,y)=0$.

\item $\rho^{\prime}\circ F=\rho^{\prime}$, as $\rho^{\prime}(x,m)_{\tau}=\rho(x).$

\item $\sigma\circ F=\sigma$ and $F\circ i=i$, which follows from the definition of $F$.
\end{enumerate}
Hence, the map $F:(\mathcal{L}^{\prime},\alpha^{\prime}) \rightarrow (\mathcal{L}^{\prime},\alpha^{\prime})$ is an automorphism of hom-Lie Rinehart algebra $(\mathcal{L}^{\prime},\alpha^{\prime})$.
\end{proof}

\begin{Rem}
In Section $2$ of \cite{Harr}, a similar result appeared for an associative and commutaive algebra $A$ by showing that  the first Harrison cohomology space of $A$ with coefficients in a module $M$ can be interpreted as the set of automorphisms of any given extension of $A$ by $M$. Also, in Chapter $7$ of \cite{Weibel}, the space of derivations of a Lie algebra $L$ into a $L$-module $M$, is identified by the space of automorphisms of trivial extensions of $L$ by $M$ or in other words by a subgroup of $Aut(L\rtimes M)$, consisting of those automorphisms which stabilize $L$ and $M$. Note that for a Lie algebra $L$, first Chevalley-Eilenberg cohomology space with coefficients in module $M$ is the space of outer derivations of $L$ into $M$, but here $H^1_{hLR}(L,M)\cong Der(L,M)$.
\end{Rem}

\subsection{Central Extensions:}
Define the center of a hom-Lie-Rinehart algebra $(\mathcal{L},\alpha)$ by
$$Z_A (\mathcal{L})=\{x\in L : [a.x,z]=[a.\alpha(x),z]=0,~\mbox{and} ~x(a)=0~~\mbox{for all}~a\in A,~z\in L \}.$$
 \begin{Def}
A short exact sequence of hom-Lie-Rinehart algebras 
$$\begin{CD}
(\mathcal{M},\beta)@>i>> (\mathcal{L}^{\prime},\alpha^{\prime}) @> \sigma>>(\mathcal{L},\alpha)
\end{CD}$$
is called central extension of $(\mathcal{L},\alpha)$ if $i(M)=Ker(\sigma)\subset Z_A(\mathcal{L}^{\prime}) $. Here, $(\mathcal{M},\beta)$ is a hom-Lie-Rinehart algebra $(A,M,[-,-]_M,\phi,\beta,\rho_M)$. 
\end{Def}
\begin{Rem}\label{Cent}
\begin{enumerate}
\item Since $\sigma\circ i=0$, we have $\rho_M=0$.
\item Note that $i(M)=Ker(\sigma)\subset Z_A(\mathcal{L}^{\prime})$ implies $(\mathcal{M},\beta)$ is a hom-Lie-Rinehart algebra with trivial bracket. Since $i[m,n]_M=[i(m),i(n)]^{\prime}=0$ and $i$ is an injective
map.
\item If $(M,\beta)$ is a trivial hom-Lie-Rinehart algebra module over $(\mathcal{L},\alpha)$, then an abelian extension of $(\mathcal{L},\alpha)$ by the module $(M,\beta)$ is a central extension.  
\end{enumerate}
\end{Rem}
\begin{Prop}\label{EQC}
There is a one-to-one correspondence between the equivalence classes of $A$-split central extensions
$$\begin{CD}
(\mathcal{M},\beta)@>i>> (\mathcal{L}^{\prime},\alpha^{\prime}) @> \sigma>>(\mathcal{L},\alpha)
\end{CD}$$
of $(\mathcal{L},\alpha)$ by $(\mathcal{M},\beta):=(A,M,[-,-]_M,\phi,\beta,\rho_M)$ and the cohomology classes in $H^2_{hLR}(L,M)$, where $(M,\beta)$ is a trivial hom-Lie-Rinehart algebra module over $(\mathcal{L},\alpha)$ .
\end{Prop}
\begin{proof}
Here we follow the Remark \ref{Cent}  $(3)$ and we get that $H^2_{hLR}(L,M)$ classifies the $A$-split central extensions of $(\mathcal{L},\alpha)$ by $(\mathcal{M},\beta)$.
\end{proof}
\begin{Rem}
If we consider $\alpha=Id_L$, then the category $hLR^{\phi}_A$ is  the category of Lie-Rinehart algebras over $A$. In this case, Proposition \ref{EQC} gives the correspondence between the isomorphism class of $A$-split central extensions of a Lie-Rinehart algebra $L$ over $A$ by another Lie-Rinehart algebra $M$ over $A$ and the $2$-nd cohomology space $H^2_{Rin}(L,M)$ with coefficients in trivial module $M$. ( This is given in detail in \cite{CentExt} ). 

\end{Rem}

\section{hom-Lie-Rinehart algebras associated to Poisson algebra }
Let A be an $R$-algebra and $\phi$ be an $R$-algebra automorphism. In Definition \ref{DerS}, the notion of $\phi$-derivations of $A$ into an $A$-module $M$ is defined. Here, we define the universal property of $\phi$-derivations and prove the existence of a universal $\phi$-derivation for an $R$-algebra $A$ with an $R$-algebra automorphism. The following definitions and results are obtained from the discussions in classical algebra with derivation (as in  \cite{Kunz}) and from the isomorphism of $A$-modules $Der_{\phi}(A,M)\cong Der_R(A,M)$ (given by $d\mapsto d\circ \phi^{-1}$).  

\begin{Prop}\label{CorrDER}
Let $M$ be an $A$-module, $A\ltimes M$ be the semi-direct product and $d:A\rightarrow M$ be a $\phi$-derivation, then the map $\tilde{d}:A\rightarrow A\ltimes M$ given by 
$$\tilde{d}(a)=(a,d\phi^{-1}(a)) ~\mbox{for} ~a\in A,$$
is an $R$-algebra homomorphism with $\pi_1\circ \tilde{d}=Id_A$, where $\pi_1(a, m)=a$.

Conversely, for an $R$-algebra homomorphism $h:A\rightarrow A\ltimes M$ satisfying $\pi_1\circ \tilde{d}=Id_A$, there is a unique $\phi$-derivation $d:A\rightarrow M$ with $h=\tilde{d}$. 
\end{Prop}
\begin{proof}
Let $a,b\in A$. From the definition of $\tilde{d}$ and $\pi$ we get 
$$\tilde{d}(ab)=(ab,d\phi^{-1}(ab))=(a,d\phi^{-1}(a))(b,d\phi^{-1}(b))=\tilde{d}(a)\tilde{d}(b),$$
and $\pi_1\circ \tilde{d}=Id_A$.

Conversely, let $h:A\rightarrow A\ltimes M$ be an $R$-algebra homomorphism satisfying $\pi_1\circ \tilde{d}=Id_A$. Then for $a\in A$, we can write 
$h(a)=(a,h_1(a))$
where $h_1:A\rightarrow M$ is an $R$-linear map. Define $d:A\rightarrow M$ by
$$d(a)=h_1(\phi(a)) ~\mbox{for} ~a\in A.$$ 

It follows that $d$ is a $\phi$-derivation and the map $\tilde{d}=h$.
\end{proof}
\begin{Def}
Let $M$ be an $A$-module. A $\phi$-derivation $d:A\rightarrow M$ is called universal, if for a given $A$-module $N$ and a $\phi$-derivation $\delta : A\rightarrow N$, there is a unique $A$-module homomorphism $f:M\rightarrow N$ with $\delta=f\circ d$.
\end{Def}
Let $M_1,~M_2$ be $A$-modules. If $d_1:A\rightarrow M_1$ and $d_2: A\rightarrow M_2$ are universal $\phi$-derivations, then it follows that $M_1$ and $M_2$ are isomorphic as $A$-modules, so the universal $\phi$-derivation of $R$-algebra $A$ is unique up to isomorphism. Next, we describe the existence of such a universal $\phi$-derivation.
\begin{Thm}
For any $R$-algebra $A$ with an $R$-algebra automorphism $\phi$, there exists a universal $\phi$-derivation.
\end{Thm}
\begin{proof}
Let $\mu:A\otimes A\rightarrow A$ be the multiplication on $A$. Then $I:=Ker \mu$ is generated by the elements of the form $\phi(a)\otimes 1-1\otimes \phi(a)$ where $a\in A$. Note that $A\otimes A$ has an $A$-module structure, given by the ring homomorphism $\psi:A\rightarrow A\otimes A$ defined by $\psi(a)=a\otimes 1$. Now, consider the map
$d: A\rightarrow I/I^2$ given by
 $d(a)=\phi(a)\otimes 1-1\otimes \phi(a)+I^2 ~\mbox{for}~a\in A. $
Then it follows that $d:A\rightarrow I/I^2$ is a $\phi$-derivation.

Let $N$ be an $A$-module and $\delta: A\rightarrow N$ be a $\phi$-derivation. By Proposition \ref{CorrDER}, we have an $R$-algebra homomorphism $\tilde{\delta}:A\rightarrow A\ltimes N$ given by $\tilde{\delta}(a)=(a,\delta\phi^{-1}(a))$ for $a\in A$. If we consider the map $i: A\rightarrow A\ltimes N$ defined by $i(a)=(a,0)$, then we get an $R$-algebra homomorphism $h:A\otimes A\rightarrow A\ltimes N$ defined by
$$h(a\otimes b)=\tilde{\delta}(a)i(b)=(ab,b\delta\phi^{-1}a); $$
for all $a,b\in R$. So, for $x=\phi(a)\otimes 1-1\otimes \phi(a)\in I$, we have 
$h(x)=(0,\delta(a)).$ 

Thus $h$ vanishes on $I^2\subset A\otimes A $. So it  induces an $A$-module homomorphism, say $\tilde{h}: I/I^2\rightarrow N$ such that $\tilde{h}(da)=\delta(a)$ for all $a\in A$. Consequently, the map $d:A\rightarrow I/I^2$ is a universal $\phi$-derivation.
\end{proof}

We now consider the $A$-module $D_A^{\phi}$ generated by the symbols $da$, for $a\in A$ subject to the following relations:
\begin{enumerate}
\item $d(\lambda a+\mu b)=\lambda da+\mu db$;
\item $d(ab)=\phi(a)db+ \phi(b)da$;
\end{enumerate}
for all $a,b\in A,~\mbox{and}~\lambda, \mu\in R$. Then the map $d:A\rightarrow D_A^{\phi}$ is a universal $\phi$-derivation. Hence by universal property $D_A^{\phi}\cong I/I^2$.

Let us call $D_A^{\phi}$ the $A$-module of formal $\phi$-differentials of the pair $(A,\phi)$. Observe that for any $\phi$-derivation $d_1:A\rightarrow M$, one gets a derivation $\delta_1=d_1\circ \phi^{-1}:A\rightarrow M$. So, there is a correspondence between a universal $\phi$-derivation $d:A\rightarrow D_A^{\phi}$ and the universal derivation $\delta: A\rightarrow D_A$, where $D_A$ is the module of K\"ahler differentials. Also, $D_A^{\phi}\cong D_A$ as $A$-modules. Moreover, the universal property of the $\phi$-derivation $d:A\rightarrow D_A^{\phi}$ gives the following result.

\begin{Prop}\label{PDerIso}
Let  $d:A\rightarrow D_A^{\phi}$ is a universal $\phi$-derivation and $M$ be an $A$-module, then the canonical map $\Phi: Hom_A(D^{\phi}_A,M)\rightarrow Der_{\phi}(A,M)$ given by $\Phi(f)=f\circ d$ where $f\in Hom_A(D^{\phi}_A,M) $,  is an $A$-module isomorphism. In the particular case if $M=A$ then we have a canonical isomorphism of $A$-modules 
$$ Hom_A(D_A^{\phi},A)\cong Der_{\phi}(A).$$ 
\end{Prop}
Let $(A,\mu,[-,-]_A)$ be a Poisson algebra and $\phi:A\rightarrow A$ be a Poisson algebra automorphism. Then $(A,\mu,\{-,-\},\phi)$ is a purely hom-Poisson algebra, where $\{-,-\}=\phi\circ[-,-]_A$. And the map $\{a,-\}:A\rightarrow A$ is a $\phi$-derivation of $A$. Thus we obtain an alternating $A$-bilinear map $$\pi:D_A^{\phi}\otimes_A D_A^{\phi}\rightarrow A $$ given by
$$\pi(adx,bdy)=ab\{x,y\}$$
for all $x,y,a,b\in A$. The map $\pi$ induces an $A$-linear map $\pi^{\prime}:D_A^{\phi}\rightarrow Hom_A(D_A^{\phi},A)$. Also, by using Proposition \ref{PDerIso}, we obtain an $A$-linear map 
\begin{equation*}
\pi^{*}:D_A^{\phi}\rightarrow Der_{\phi}(A)
\end{equation*}
defined by $\pi^{*}(dx)(a)=\pi(dx,da)~~\mbox{for all}~a,x\in A.$ 

Now consider a bilinear map $[-,-]:D_A^{\phi}\otimes D_A^{\phi}\rightarrow D_A^{\phi}$ given by
 $$[adx,bdy]=\phi(a)\phi(b)d\{x,y\}+\phi(a)\{\phi(x),b\}d\phi(y)-\phi(b)\{\phi(y),a\}d\phi(x)$$
 and define a linear map $\tilde{\alpha_{\phi}}:D_A^{\phi}\rightarrow D_A^{\phi}$ given by 
$$\tilde{\alpha_{\phi}}(adx)=\phi(a)d(\phi(x)).$$
Then the resulting tuple  $(D_A^{\phi},[-,-],\tilde{\alpha_{\phi}})$ is a hom-Lie algebra; where the required hom-Jacobi identity for the bracket boils down to the following equation
$$[[dx,dy],d\phi(z)]+[[dy,dz],d\phi(x)]+[[dz,dx],d\phi(y)]=0,$$
which follows from the hom-Jacobi identity for the bracket $\{-,-\}$. If we now consider the map $\rho_{\phi}:D_A^{\phi}\rightarrow Der_{\phi}(A)$ defined by $$\rho_{\phi}=\pi^{*}\circ \tilde{\alpha_{\phi}},$$ then it follows that $(\rho_{\phi},\phi)$ is a representation of the hom-Lie algebra $(D_A^{\phi},[-,-],\tilde{\alpha_{\phi}})$ on $A$. Furthermore, for any $\xi_1,\xi_2\in D_A^{\phi}$ and $a\in A$, we can deduce that
$$[\xi_1,a.\xi_2]=\phi(a)[\xi_1,\xi_2]+\rho_{\phi}(\xi_1)(a).\tilde{\alpha_{\phi}}(\xi_2).$$ 
As a result, the above discussion gives us a hom-Lie-Rinehart algebra structure on the $A$-module of $\phi$-differentials.
\begin{Thm}\label{ConstExm}
Let $(A,\mu,[-,-]_A)$ be a Poisson algebra , $\phi:A\rightarrow A$ be a Poisson algebra automorphism and
 $(A,\mu,\{-,-\},\phi)$ be the purely hom-Poisson algebra, obtained by composition then $(A,D_A^{\phi},[-,-],\phi,\tilde{\alpha_{\phi}},\rho_{\phi})$ is a hom-Lie-Rinehart algebra over $(A,\phi)$ where
\begin{enumerate}
\item the endomorphism $\tilde{\alpha_{\phi}}:D_A^{\phi}\rightarrow D_A^{\phi}$ is given by $\tilde{\alpha_{\phi}}(adx)=\phi(a)d(\phi(x));$
\item the anchor map $\rho_{\phi}:D_A^{\phi}\rightarrow Der_{\phi}(A)$ is given by $\rho_{\phi}=\pi^{*}\circ \tilde{\alpha_{\phi}};$
\item the bracket is given by
 $$[adx,bdy]=\phi(a)\phi(b)d\{x,y\}+\rho_{\phi}(adx)(b)\tilde{\alpha_{\phi}}(dy)-\rho_{\phi}(bdy)(a)\tilde{\alpha_{\phi}}(dx);$$
\end{enumerate} 
for all $a,b,x,y\in A$.
\end{Thm}
\begin{Rem}
In particular, if $A$ is a Poisson algebra and $\phi=Id_A$, then the above construction gives a Lie-Rinehart algebra structure on $D_A$, the $A$-module of formal differentials (or the module of \textbf{K\"ahler differentials}). This is given by J. Huebschmann in \cite{Hueb1}. 
\end{Rem}

\begin{Rem}
The hom-Lie-Rinehart algebra appeared in Theorem \ref{ConstExm} has an invertible endomorphism since $\phi:A\rightarrow A$ is a Poisson algebra automorphism. This kind of hom-Lie-Rinehart algebras are called regular hom-Lie-Rinehart algebras in \cite{New} and a new cochain complex is defined for the regular case,  in which the bilinear map $\pi:D_A^{\phi}\otimes D_A^{\phi}\rightarrow D_A^{\phi}$
is a $2$-cocycle.
\end{Rem}
We can rewrite the hom-Lie bracket on $D_A^{\phi}$ in terms of Lie derivatives.
First let us define an operator $L_X:D_A^{\phi}\rightarrow D_A^{\phi}$  as
\begin{equation}\label{LieDerivative}
L_X(adf)=\phi(a)d(X(f))+\alpha_{\phi}(X)(a)d(\phi(y))~~\mbox{for}~X\in Der_{\phi}(A).
\end{equation}
Here, the map $\alpha_{\phi}$ is anchor map of the hom-Lie-Rinehart algebra $(A,Der_{\phi}A,[-,-]_{\phi},\alpha_{\phi},\alpha_{\phi})$ given in Example \ref{pderhl}. In terms of this operator $L_X:D_A^{\phi}\rightarrow D_A^{\phi}$,  the bracket in Theorem \ref{ConstExm} can be rewritten as:
$$[\xi_1,\xi_2]=L_{\pi^*(\xi_1)}(\xi_2)-L_{\pi^*(\xi_2)}(\xi_1)-d\pi(\xi_1,\xi_2);$$
for any $\xi_1,\xi_2\in D_A^{\phi}$. We say this operator $L_X:D_A^{\phi}\rightarrow D_A^{\phi}$ to be the Lie derivative with respect to $X\in Der_{\phi}(A)$ . 

\section{Special cases of hom-Lie-Rinehart and hom-Gerstenhaber algebras}
We present a discussion on some special cases of hom-Lie-Rinehart algebra and hom-Gerstenhaber algebra, which show the wide interests and further application of these algebras.

\subsection{Hom-Lie algebras:} Considering a hom-Lie algebra in the category of hom-Lie-Rinehart algebras we get the following:

\textbf{1. Deformation cohomology of a Hom-Lie algebra:}\\
 Let us consider $R$ to be a field. The hom-Lie algebra $(L,[-,-],\alpha)$ is a hom-Lie module over itself by adjoint action. Note that $(L,[-,-],\alpha)$ is also a hom-Lie-Rinehart algebra (Example \ref{hom-Lie}) and a hom-Lie-Rinehart module over itself as the action on $R$ is trivial. Then the cohomology $H^*_{hLR}(L,L)$ is same as the deformation cohomology for a hom-Lie algebra defined in \cite{DefHLIE}. 

\textbf{2.Extensions of a Hom-Lie algebra:} \\
In Section $2.4$, \cite{HLIE01}, A construction of central extension of a Hom-Lie algebra is given. After constructing a certain q-deformation of the Witt algebra, Hartwig, Larsson, and Silvestrov used the machinery of Hom-Lie algebra extensions to construct a corresponding deformation of the Virasoro algebra in Section 4, \cite{HLIE01}. In the classical case of Lie algebras, equivalence classes of abelian (and central) extensions are classified by the 2nd cohomology module \cite{HomAlg}. For hom-Lie algebras, classification of split abelian extensions and central extension can be done by considering them in the category of hom-Lie-Rinehart algebras and then using Theorem \ref{Char1} and Proposition \ref{EQC}.

A short exact sequence 
$$\begin{CD}
0 @>>>(M,0,\beta) @>i>> (L^{\prime},\{-,-\},\alpha^{\prime}) @>\sigma>> (L,[-,-],\alpha)@>>> 0.
\end{CD}$$
is called an extension of the hom-Lie algebra $(L,[-,-],\alpha)$ by the abelian hom-Lie algebra $(M,0, \beta)$. 

Here, by an split extension of hom-Lie algebra $(L,[-,-],\alpha)$, we mean that we have a $R$-linear map $\tau: L \rightarrow L^{\prime}$ such that
\begin{enumerate}
\item $\sigma \circ \tau=Id_L$, and
\item $\tau\circ\alpha = \alpha^{\prime}\circ \tau$,

\end{enumerate}
for each $a\in A,~x\in L$. Furthermore, if the section $\tau$ for $\sigma$ is a homomorphism of hom-Lie algebras then this extension is called split in the category of hom-Lie algebras. Now, one can deduce the following result:
\begin{Prop}
There is a one-to-one correspondence between the equivalence classes of split abelian extensions of a hom-Lie algebra $(L,[-,-],\alpha)$ by $(M,0,\beta)$ and the cohomology classes in $H^2_{hLR}(L,M)$.
\end{Prop}
Similarly, in the case: $L$ acting trivially on $M$, one can characterize the second cohomology space $H^2_{hLR}(L,M)$ by equivalence classes of split central extensions of the hom-Lie algebra $(L,[-,-],\alpha)$ by $(M,0,\beta)$ since if $L$ acts trivially on $M$, then any central extension of $(L,[-,-],\alpha)$ by $(M,0,\beta)$ is an abelian extension. 
 
\subsection{Lie-Rinehart algebras}
Let $(\mathcal{L},\alpha)$ be a hom-Lie-Rinehart algebra then the pair $(A,\phi)$ is a left $(\mathcal{L},\alpha)$-module. If $\alpha=Id_L$, then hom-Lie- Rinehart algebra $(\mathcal{L},\alpha)$ is a Lie-Rinehart algebra $L$ over $A$ and $A$ is a left Lie-Rinehart algebra module over $L$. The cochain complex $(C^*(L,A),\delta)$ is same as the Lie-Rinehart algebra cochain complex except $0$-cochains and for $n\geq 2$ the $n^{th}$-cohomology space $H^n_{hLR}(L,A)$ of $(\mathcal{L},\alpha)$ with coefficients in $(A,\phi)$ is same as the $n^{th}$-cohomology space of Lie-Rinehart algebra with coefficients in the module $A$. Hence by Theorem \ref{Char1}, we get the characterisation of second Lie-Rinehart algebra cohomology space in terms of $A$-split abelian extensions, which is proved in Section 2, \cite{Hueb1}.

\subsection{hom-Lie algebroids}

We have defined the notion of a left module over a hom-Lie-Rinehart algebra. The algebraic properties of the left modules over hom-Lie-Rinehart algebra motivates the following definition of a representation of a hom-Lie algebroid as follows:

\begin{Def}
Let $\mathcal{A}:=(A\rightarrow M, [-,-],\alpha, \phi,\rho)$ be a hom-Lie algebroid over $M$. A representation of $\mathcal{A}$ on a vector bundle $E\rightarrow M $ is a triplet $(E\rightarrow M, \beta,\nabla^{\phi^*})$, where $\beta: \Gamma E\rightarrow \Gamma E$ is a linear map and $\nabla^{\phi^*}: \Gamma A\otimes \Gamma E\rightarrow \Gamma E$ is a bilinear map, given by $(x,s)\mapsto\nabla^{\phi^*}_x(s)$, satisfying the following properties:  
\begin{enumerate}
\item $\beta(f.s)=\phi^*(f).\beta(s)$;
\item $\nabla^{\phi^*}_{f.x}(s)=\phi^*(f).\nabla^{\phi^*}_x(s)$;
\item $\nabla^{\phi^*}_x(f.s)=\phi^*(f).\nabla^{\phi^*}_x(s)+ \rho(x)[f].\beta(s)$;
\item $(\nabla^{\phi^*},\beta)$ is a representation of hom-Lie algebra $(\Gamma A,[-,-],\alpha)$ on $\Gamma E$;
\end{enumerate}
for all $x\in \Gamma A,~s\in \Gamma E$ and $f\in C^{\infty}(M)$. 
\end{Def}

If $\alpha=Id_{\Gamma A}$, and $\beta=Id_{\Gamma E}$, then $\mathcal{A}$ is a Lie-algebroid and the triple $(E\rightarrow M, \beta,\nabla^{\phi^*})$ is the usual representation of a Lie-algebroid on the vector bundle $E\rightarrow M$.

For a hom-Lie algebroid $\mathcal{A}:=(A\rightarrow M, [-,-],\alpha, \phi,\rho)$, there is a canonical representation $(M\times \mathbb{R}\rightarrow M, \phi^*,\nabla^{\phi^*})$ on the trivial bundle $M\times \mathbb{R}\rightarrow M$, where $\nabla^{\phi^*}:\Gamma A\otimes C^{\infty}(M)\rightarrow C^{\infty}(M)$ is given by
$$\nabla^{\phi^*}_x(f)=\rho(x)[f]~\mbox{for any}~x\in \Gamma A,~ f\in C^{\infty}(M).$$
A coboundary operator for a hom-Lie algebroid with coefficients in it's representation on a vector bundle can be defined since any hom-Lie algebroid is a hom-Lie-Rinehart algebra. Now, let us first recall the definition of $(\sigma,\tau)$-differential graded commutative algebra  from \cite{OnhLie}.
\begin{Def}
Let $\mathfrak{A}=\oplus_{k\in\mathcal{Z}_+}\mathfrak{A}_k$ be a graded commutative algebra, $\sigma$ and $\tau$ be $0$-degree endomorphism of $\mathfrak{A}$, then a $(\sigma,\tau)$-differential graded commutative algebra is quadruple $(\mathfrak{A},\sigma,\tau,d)$, where $d$ is a degree $1$ square zero operator on $\mathfrak{A}$ satisfying the following:
\begin{enumerate}
\item $d\circ\sigma=\sigma\circ d,~ ~d\circ\tau=\tau\circ d;$
\item $d(ab)=d(a)\tau(b)+(-1)^{|a|}\sigma(a)d(b)$ for $a,b\in \mathfrak{A}$.
\end{enumerate}
 
\end{Def}
In \cite{New}, for a regular hom-Lie algebroid $\mathcal{A}$, we define a different cochain complex $(C^*(\mathcal{A};E),d_{A,E})$ with coefficients in a representation $(E\rightarrow M, \beta,\nabla^{\phi^*})$. In the case of the trivial representation of $\mathcal{A}$ on $M\times \mathbb{R}\rightarrow M$, the cochain complex $(C^*(\mathcal{A};M\times \mathbb{R}),d_{A,M\times \mathbb{R}})$ is a $(\tilde{\alpha},\tilde{\alpha})$-differential graded commutative algebra, where the map $\tilde{\alpha}:\Gamma(\wedge^n A^*)\rightarrow \Gamma(\wedge^{n}A^*)$ is defined as
$$\tilde{\alpha}(\xi)(x_1,\cdots,x_n)=\phi^*(\xi(\alpha^{-1}(x_1),\cdots,\alpha^{-1}(x_n)))$$
for $\xi\in \Gamma(\wedge^nA^*),~\mbox{and}~ x_i\in \Gamma A, ~\mbox{for}~ 1\leq i\leq n$.

Note that for a Lie algebroid ($\alpha=Id_{\Gamma A}$, and $\phi=Id_{M}$) this complex with trivial coefficients becomes the de Rham complex and the $(\tilde{\alpha},\tilde{\alpha})$-differential graded commutative algebra becomes the usual differential graded commutative algebra associated to a Lie algebroid. Furthermore, this differential $d_{A,M\times \mathbb{R}}$ allows us to define the notion of Lie derivative for hom-Lie algebroids. In \cite{hom-Lie1}, the authors have defined the notion of Lie derivative by obtaining a $(\sigma,\tau)$-differential graded commutative algebra associated to a regualar (or invertible) hom-Lie algebroid, but for a modified definition of a hom-Lie algebroid (originally given in \cite{hom-Lie}). Both definitions are equivalent in the invertible case, however  it is important to note that by Remark \ref{CompS} such modification has a disadvantage in the non-invertible case. In \cite{New}, the results in Section 5 show that the notion of Lie derivative can be defined without a modification.

\subsection{hom-Batalin-Vilkovisky algebras}
A Gerstenhaber algebra with an exact generator is called a \textbf{Batalin- Vilkovisky algebra} (see \cite{Hueb3,Xu} and references therein). In \cite{Hueb3}, J. Huebschmann relates Gerstenhaber structures and homology and cohomology of Lie-Rinehart algebras with new insights from the notion of a Batalin-Vilkovisky algebras. The discussion of such relation in hom-Lie-Rinehart context may be considered by exact generators of the associated hom-Gerstenhaber algebra.

In \cite{Xu}, Ping Xu establised a correspondence between various geometric structures on vector bundles and the algebraic structures such as Gerstenhaber algebras and Batalin-Vilkovisky algebras. In \cite{hom-Lie}, the authors given a canonical hom-Gerstenhaber algebra structure on the exterior algebra $(\mathfrak{G}=\wedge^*\mathfrak{g},\wedge,[-,-]_{\mathfrak{G}},\alpha_{\mathfrak{G}})$ associated to a hom-Lie algebra $(\mathfrak{g},[-,-],\alpha)$( see Example \ref{Hom-Ger2}). If we consider the boundary operator $d$ in the complex for a hom-Lie algebra ( with coefficients in the trivial module $R$ in \cite{HALG2}) then the map $d:\wedge^n\mathfrak{g}\rightarrow \wedge^{n-1}\mathfrak{g}$ is given by 
$$d(x_1\wedge\cdots\wedge x_n)=\sum_{1\leq i<j\leq n} (-1)^{i+j}[x_i,x_j]\wedge \alpha_G(x_1\wedge\cdots \hat{x_i}\wedge\cdots\wedge \hat{x_j}\wedge\cdots\wedge x_n)$$
for all $x_1,\cdots,x_n\in \mathfrak{g}$. In other words, $d:\mathfrak{G}\rightarrow \mathfrak{G}$ is a map of degree $-1$ such that $d^2=0$. More importantly, this operator $d$ generates the graded hom-Lie bracket $[-,-]_{\mathfrak{G}}$ in the following way:

\begin{equation}
[X,Y]_{\mathfrak{G}}=(-1)^{|X|}(d(XY)-(dX)\alpha_{\mathfrak{G}}(Y)-(-1)^{|X|}\alpha_{\mathfrak{G}}(X)(dY));
\end{equation}  
for $X,Y\in \mathfrak{G}$. We may take the operator $d$ as an exact generator of the hom-Gerstenhaber algebra $\mathfrak{G}$. This yields a hom-Gerstenhaber algebra with an exact generator which is a \textbf{\it hom-Batalin-Vilkovisky algebra}. We would like to elaborate more on hom-Lie-Rinehart, hom-Gerstenhaber and hom-Batalin-Vilkovisky algebras in a separate note. Also, by using the notion of generators of hom-Gerstenhaber algebras a homology may be associated to a hom-Lie algebroid which is analogous to that for Lie algebroid in \cite{Xu}. We have seen an example of hom-Gerstenhaber algebra with an exact generator canonically associated to a hom-Lie algebra.\\
\\
\textbf{Acknowledgements:}\\

We would like to thank the anonymous referee for various comments, which improved the presentation of the paper. In particular, we thank the referee for remark on Definition \ref{Remark} and definition of the center for hom-Lie-Rinehart algebras.  


\vspace{.25cm}
{\bf Ashis Mandal and  Satyendra Kumar Mishra}\\
 Department of Mathematics and Statistics,
Indian Institute of Technology,
Kanpur 208016, India.\\
e-mail: amandal@iitk.ac.in, ~~ satyendm@iitk.ac.in

\end{document}